\documentclass[11pt]{article}
\usepackage{amsmath, amsthm, amssymb,mathtools}
\usepackage[sort]{cite}

\usepackage{geometry}
\usepackage{hyperref}

\usepackage{enumitem}

\title{Large friction limit of the almost pressureless Euler-Poisson system}

\author{
Xin Liu\thanks{Texas A\&M University, Department of Mathematics, College Station, TX 77843, USA. Email: \href{mailto:xliu23@tamu.edu}{xliu23@tamu.edu}}}

\date{\today}

\theoremstyle{plain}
\newtheorem{theorem}{Theorem}[section]

\newtheorem{corollary}[theorem]{Corollary}

\theoremstyle{definition}
\newtheorem{definition}[theorem]{Definition}

\theoremstyle{remark}
\newtheorem{remark}[theorem]{Remark}

\newcommand{\mrm}[1]{\mathrm{#1}}

\newcommand{\dt}{\partial_t}
\newcommand{\dx}{\partial_x}
\newcommand{\dtau}{\partial_\tau}
\newcommand{\dv}{\mathrm{div}\, }
\renewcommand{\vec}[1]{{\bf #1}}

\newcommand{\idx}{\ d\vec x}
\newcommand{\norm}[2]{\Vert #1 \Vert_{#2}}

\allowdisplaybreaks
\numberwithin{equation}{section}

\begin{document}

\maketitle

\begin{abstract}
The goal of this work is to investigate the almost pressureless Euler-Poisson (EP) system with repulsive force in the large friction limit. The leading order equations in the limit are shown to be the hyperbolic-elliptic Keller-Segel (KS) system of consumption type. Under suitable assumptions on the initial data, we establish the unique global-in-time solutions to both the EP system and the KS system by establishing the global stability in the large friction limit. In particular, no singularity forms in the asymptotic limit. Moreover, the time asymptotic behavior of the one-dimensional KS flow with vacuum is also discussed. 

\smallskip 


\smallskip

\end{abstract}

\tableofcontents

\section{Introduction}

\subsection{The systems under consideration and motivation}

Let $ \rho \in \mathbb R $, $\vec u \in \mathbb R^3$, and $ \phi \in \mathbb R $ be the density, velocity, and the electric potential, respectively. 
The almost pressureless Euler-Poisson (EP) system with large friction is given by the following system of equations: 
\begin{subequations}
        \label{sys:PLEP}
    \begin{align}
        \label{eq:PLEP-1}
        \dt \rho + \dv (\rho \vec u) & =0, & \text{in} \ \Omega, \\
        \label{eq:PLEP-2} 
        \dt (\rho \vec u) + \dv (\rho \vec u \otimes \vec u) + \varepsilon^\alpha \nabla p + \frac{1}{\varepsilon} \rho \vec u + \rho \nabla \phi & = 0, & \text{in} \ \Omega, \\
        \label{eq:PLEP-3}
        - \Delta \phi & = \rho - M, & \text{in} \ \Omega,
    \end{align}
    where $ p = \rho^\gamma $, $ \gamma > 1 $ is the pressure potential.
\end{subequations} 
Here $ \alpha > 0 $ is a fixed parameter. $ M > 0 $ is the constant background charge. $ 0 < \varepsilon \ll 1 $ is a small parameter. As $ \varepsilon \rightarrow 0 $, system \eqref{sys:PLEP} is said to be in the large friction, almost pressureless regime. 

System \eqref{sys:PLEP} models the dynamics of a plasma (charged particles) in an electric field generated by the background charge $ M $. It should be noted that system \eqref{sys:PLEP} is the Euler-Poisson system with repulsive force $\rho \nabla\phi$. A similar system with attracting force, i.e., $ - \rho \nabla \phi$ in \eqref{eq:PLEP-2}, is also of great importance, and is used to model the dynamics of self-gravitation systems, such as gaseous stars. 

Furthermore, consider the time scale $ t \simeq 1/\varepsilon $ by introducing the following rescaling of time:
\begin{equation}
    \label{eq:rescale-time}
    t =  \tau/\varepsilon.
\end{equation}
Then the system \eqref{sys:PLEP} can be rewritten as:
\begin{subequations}
    \label{sys:PLEP-rescaled}
    \begin{align}
        \label{eq:rsPLEP-1}
        \dtau \rho + \frac{1}{\varepsilon}\dv (\rho \vec u) & =0, & \text{in} \ \Omega , \\
        \label{eq:rsPLEP-2} 
        \dtau (\rho \vec u) + \frac{1}{\varepsilon}\dv (\rho \vec u \otimes \vec u) + \frac{1}{\varepsilon^{1-\alpha}} \nabla p + \frac{1}{\varepsilon^2} \rho \vec u + \frac{1}{\varepsilon}\rho \nabla \phi & = 0, & \text{in} \ \Omega , \\
        \label{eq:rsPLEP-3}
        - \Delta \phi & = \rho - M, & \text{in} \ \Omega .
    \end{align}
\end{subequations}

\smallskip 

Formally, one can check that the leading order dynamics of system \eqref{sys:PLEP-rescaled} as $ \varepsilon \rightarrow 0 $ is given by the following hyperbolic-elliptic Keller-Segel (KS) system of consumption type: 
\begin{subequations}
    \label{sys:PLEP-limit}
    \begin{align}
        \label{eq:rsPLEPlt-1}
        \dtau \sigma - \dv (\sigma \nabla \psi) & =0, & \text{in} \ \Omega \times (0,\infty), \\
        \label{eq:rsPLEPlt-3}
        - \Delta \psi & = \sigma - M, & \text{in} \ \Omega \times (0,\infty),
    \end{align}
    or equivalently,
    \begin{equation}{\tag{\ref{sys:PLEP-limit}'}}
        \label{eq:rsPLEPlt-limit}
        \dtau \sigma + \vec v \cdot \nabla \sigma + \sigma (\sigma-M) = 0, \quad \text{with} \quad \vec v = - \nabla (-\Delta)^{-1}(\sigma-M).
    \end{equation}
    where formally, $ \sigma = \lim_{\varepsilon\to 0} \rho $ and $ \psi = \lim_{\varepsilon\to 0} \phi $. 
    \end{subequations}
    Here $ \sigma $ is referred to as the bacteria density, and $ \psi $ is referred to as the chemical potential for the nutrition in the literature. System \eqref{sys:PLEP-limit} is the Keller-Segel system of consumption type. A similar system with opposite sign in front of the divergence in \eqref{eq:rsPLEPlt-1} is referred to as the KS system of production type. 

    \smallskip 

    For suitable and smooth initial data, the local well-posedness of both system \eqref{sys:PLEP-rescaled} (fixed $ \varepsilon $) and system \eqref{sys:PLEP-limit} is classical. However, for general initial data, the classical solutions to both systems are expected to blow up in finite time. The case with attractive force/production type are similar. We defer the review of the literatures in section \ref{sec:background}, below.
    
    Meanwhile, for the long time behavior of solutions to \eqref{eq:rsPLEPlt-limit}, it is easy to see that as long as the $ \sigma > 0 $, i.e., $ \sigma $ is non-vanishing, one can obtain a global solution. 
    In fact, $ \sigma_{e,1} := M $ is the globally stable equilibrium of \eqref{eq:rsPLEPlt-limit} and the any solution, if exists, converges to $ \sigma_{e,1} $ as $ \tau \to \infty $. Our result in this work shows that in the one space dimension, there exist global smooth solutions with generic initial data, and in the three space dimension, an additional condition for initial data is required for the system to generate global smooth solutions. On the other hand, if $ \sigma $ is vanishing in some regime, the global well-posedness is a bit subtle, since $ \sigma_{e,2} :=0 $ is an unstable equilibrium of \eqref{eq:rsPLEPlt-limit}. In this case, the global solution, if exists, may not be smooth. We will discuss the behavior of the one-dimensional flow in section \ref{sec:1-d-discussion}, below. 

    The above observation motivates the study of the high friction limit of the almost pressureless Euler-Poisson system \eqref{sys:PLEP-rescaled} with initial data $ (\rho_0, \vec u_0) $ close to the globally smooth solution of the limit system \eqref{sys:PLEP-limit}, i.e, 
    \begin{equation}
        \label{eq:initial-data}
            \rho_0 > 0, \ \vec u_0 \simeq 0.
    \end{equation}
    We defer the detailed description of the initial date to section \ref{sec:main-results}, below.
    
    \subsection{Background and literatures} 
    \label{sec:background}

    \subsubsection*{The Euler-Poisson system}

    The EP system with repulsive force consists of the conservation of mass \eqref{eq:PLEP-1}, the conservation of momentum/Newton's second law \eqref{eq:PLEP-2}, and the electric potential equation \eqref{eq:PLEP-3}. It is one of the most fundamental systems in modeling the dynamics of plasma. The EP system arises also from the non-relativistic limit of the Euler-Maxwell equations \cite{pengConvergenceCompressibleEulerMaxwell2007,yangNonrelativisticLimitEuler2010,manfrediNonrelativisticLimitsMaxwells2013}.

    Mathematically, as a hyperbolic system, it is well-known that the solutions will fail to remain smooth globally in time generally. For instance, the authors in \cite{baeFormationSingularitiesPlasma2024,baeStructureSingularitiesEulerPoisson2024,wangFormationSingularitiesEuler2014,perthameNonexistenceGlobalSolutions1990,rozanovaRepulsiveEulerPoisson2025,chaeFiniteTimeBlowup2008,wangFormationSingularitiesCompressible1998,yuenBlowupEulerEuler2011} show numerous break down scenarios in the EP system and damping/friction cannot prevent the formation of singularity. In the case of one-dimensional flow, one can establish the critical threshold for the initial data, such that the global smooth solution exists if and only if the critical threshold is not reached \cite{engelbergCriticalThresholdsEulerPoisson2001,tadmorGlobalRegularitySubcritical2008,bhatnagarCriticalThresholdsEulerPoissonalignment2023,bhatnagarCompleteCharacterizationSharp2023,luanEulerPoissonEquationsVariable2025}.

    To obtain a global smooth solution, it is then necessary to consider initial data with certain restriction/smallness. By taking advantage of the dispersive effect of the associated Klein-Gordon effect, the authors in \cite{guoSmoothIrrotationalFlows1998a,guoGlobalSmoothIon2011,germainNonneutralGlobalSolutions2013,guoAbsenceShocksOne2017,ionescuEulerPoissonSystem2013,guoGlobalSolutionsCertain2014,jangTwodimensionalEulerPoissonSystem2012} establish the global smooth solutions near the equilibrium given by $ (\rho_e = M, \vec u_e = 0 ) $ for irrotational flow in various settings. It should be emphasized that since the vorticity is transported by the flow, and no dispersive effect is available, the irrotational condition is necessary for these results. 

    To remove the irrotational condition, one will need to consider the EP system with damping/friction. In a three-dimensional exterior domain, the authors in \cite{liuGlobalSolutionsCompressible2020} establish the asymptotic stability for constant equilibrium with small perturbation.  The damping/friction provides stronger stabilizing effect than the Klein-Gordon effect. In particular, the vorticity in the damped EP system can be damped out asymptotically. 

    We refer to \cite{cordierQuasineutralLimitEulerpoisson2000,liTwofluidEulerPoisson2013,roussetTransverseAsymptoticStability2025,carrilloExistenceRadialGlobal2023,choiLargetimeBehaviorPressureless2025} and the references therein for related works for the repulsive EP system. 

    On the other hand, the EP system with attracting force for gaseous stars etc., demonstrates dramatically different global behavior. In particular, the physical vacuum becomes an important issue and it is more suitable to study the attracting EP system in the setting of the free boundary problem. Along this direction, the stability of expanding solutions is investigated in \cite{hadzicClassGlobalSolutions2019} and the collapsing solution is investigated in \cite{guoLarsonPenstonSelfsimilar2021}. We refer to \cite{Jang2009a,Jang2013,Jang2010a,Coutand2010,Coutand2011a,Coutand2012a} for earlier fundamental works in this direction. 

    Recently, \cite{chenGlobalSolutionsCompressible2024,chenGlobalSolutionsCompressible2025,chenGlobalFiniteEnergySolutions2024} investigates the existence of global weak solutions to the EP system as the vanishing viscosity limit of the Navier-Stokes-Poisson system with or without doping profile. Notably, the authors in \cite{choiLargetimeBehaviorPressureless2025} establishes the one-dimensional global solutions to the damped EP system using the method of characteristics, which is consistent with our result in the one-dimensional case. 

    This paper establishes a new kind of global smooth solutions to the repulsive EP system in the almost pressureless and large friction regime. Only the smallness of the initial velocity $\vec u\vert_{t= 0}$ in the $ H^3 $-norm and the smallness of the gradient of the initial density $ \nabla \rho\vert_{t=0} $ in the $ L^4 $-norm are required. In particular, $ \norm{\rho\vert_{t=0} - M}{H^3} $ is not required to be small, provided that no vacuum is present initially. This is in contrast to \cite{guoSmoothIrrotationalFlows1998a,ionescuEulerPoissonSystem2013,guoAbsenceShocksOne2017,liuGlobalSolutionsCompressible2020}, where both the velocity and the density perturbations are required to be small. The main mechanism of the global smooth solutions is the large damping/friction effect. It should be emphasized, despite the large friction, our problem is a singular limit problem. The damping/friction effect does not provide direct damping in the density, but through the electric force. In addition, there is no regularizing effect from the damping/friction. Moreover, as it will be shown below, the smallness of the initial velocity $ \vec u\vert_{t=0} $ is not as restricted as indicated by the associated spectrum analysis in section \ref{sec:spectrum}.

    Formally, as mentioned before, the leading effect of the EP system \eqref{sys:PLEP-rescaled} as $ \varepsilon \rightarrow 0 $, i.e., the large friction limit, is the hyperbolic-elliptic KS system of consumption type \eqref{sys:PLEP-limit}, which itself draws large attention from the mathematics community. 

    \subsubsection*{The Keller-Segel system}

    The (Patlak-)Keller-Segel system was introduced in \cite{patlakRandomWalkPersistence1953,kellerInitiationSlimeMold1970}, to describe the collective motion of cells that are attracted/repelled by a self-emitted chemical substance. Depending on the sign of the divergence in \eqref{eq:rsPLEPlt-1}, the KS system is either of the consumption type ($-$) or of the production type ($+$). 

    It is easy to see from \eqref{eq:rsPLEPlt-limit} that the solution to the production type KS system will blow up within finite time along the flow trajectory for the hyperbolic-elliptic KS system. There are hence a large amount of literatures focusing on the parabolic-elliptic or parabolic-parabolic KS system of production type, where the finite-time blow up can be suppressed by the diffusion. Indeed, there exists a critical total mass, such that solutions blow up if and only if the initial total mass is in the supercritical regime, where the diffusion is not strong enough to counterbalance the growth. See \cite{eganafernandezUniquenessLongTime2016,blanchetTwodimensionalKellerSegelModel2006,blanchetInfiniteTimeAggregation2008} for the study in two space dimensions. Similar results for the degenerate KS system can be found in \cite{sugiyamaGlobalExistenceSubcritical2006,sugiyamaApplicationBestConstant2007,blanchetCriticalMassPatlak2009}. The proof relies on the best constant for the two-dimensional Hardy-Littlewood-Sobolev inequality. A three-dimensional result on a KS type system can be found in \cite{ulusoyKellerSegelType2017}. Among these works, no blow up profile is constructed. 

    A conditional global stability result was obtained in \cite{feireislConvergenceEquilibriaKeller2007}. Recently, with small initial data, \cite{hsiehLongtimeDynamicsClassical2024} shows the global stability of the constant state. 

    On the other hand, the blow up profile for the KS system in three space dimensions is first obtained in \cite{soupletBlowupProfilesParabolic2019} and later in \cite{baiBlowupProfileKeller2025}. See \cite{liuFiniteTimeBlowup2025} for the blow up profile for the Keller-Segel-Navier-Stokes system. The study of weak solutions in the scaling invariant class can be found in \cite{kozonoExistenceUniquenessTheorem2012}. We refer to \cite{hieberStrongSolutionsKellerSegelNavierStokes2025,naGlobalWellposednessTwodimensional2024,chaeExistenceSmoothSolutions2012,winklerGlobalLargeDataSolutions2012,winklerStabilizationTwodimensionalChemotaxisNavierStokes2014,ahnGlobalClassicalSolutions2021,tanTimePeriodicStrong2021} for the study of the KS system coupled with viscous fluids, and \cite{burgerKellerSegelModel2006,iwabuchiGlobalWellposednessKeller2011,bedrossianLocalGlobalWellposedness2011,perthameTravellingPlateausHyperbolic2011,luckhausMeasureValuedSolutions2012,kiselevSuppressionChemotacticExplosion2016,takeuchiKellerSegelSystemParabolicparabolic2021,takeuchiMaximalLorentzRegularity2021,zhengGlobalClassicalSolutions2021,zhouWellposednessNonuniformDependence2021,elbarDegenerateCahnHilliardIncompressible2022,heEnhancedDissipationBlowUp2023,heIncompressibleLimitsPatlakKellerSegel2023,loanPeriodicSolutionsParabolic2024,shiEnhancedDissipationBlowup2024,takeuchiGlobalWellposednessKellerSegelNavierStokes2025,takeuchiAsymptoticBehaviorGlobal2025} for related works on the KS system.

    Meanwhile, the parabolic-elliptic KS system of consumption type is globally well-posed without vacuum as shown in \cite{gallenmullerCahnHillardKeller2024}. By using the relative entropy method, the authors in \cite{lattanzioGasDynamicsLarge2017,feireislHighFrictionLimit2024,gallenmullerCahnHillardKeller2024,elbarNonlocalEulerKortewegLocal2025} show that any entropy weak solutions of the EP with large friction is asymptotically closed to the solution to the KS system. 

    For the hyperbolic(-elliptic) KS system of production type, as we mentioned before, the solution will blow up in finite time. However, taking into account the logistic/quorum sensitivity, the solution (if exists) will remain bounded. In \cite{perthameExistenceSolutionsHyperbolic2009}, 
    a global weak solution was constructed for the KS system with logistic/quorum sensitivity. In general, however, a smooth solution will form a shock in finite time for such a system, as shown in  \cite{leeThresholdShockFormation2015,mengGlobalWellposednessBlowup2024,naFinitetimeBlowupHyperbolic2024} for both production and consumption types. Notably, the blow up is not the consequence of vacuum, or large (in the Sobolev norm) data, but purely the shock formation similar to Burgers' equation. See also \cite{zhangInitialValueProblem2022,zhouWellposednessNonuniformDependence2021,feiIllposednessHyperbolicKellerSegel2023} for related ill-posedness results. The blow up/formation of singularity of solutions to the KS system \eqref{sys:PLEP-limit}, i.e., without logistic/quorum sensitivity, is generally an open problem. In fact, we show that in the one-dimensional flow, there is no formation of singularity, in contrast to the study of \cite{leeThresholdShockFormation2015}.

    For a fully hyperbolic system, \cite{jeongWellPosednessSingularityFormation2022} shows that even for the KS system of consumption type, there is still finite time blow up for general smooth solutions. Moreover, the local existence in a weighted Sobolev space is established for the initial data with vacuum. 

    The result in this paper indicates that for the hyperbolic-elliptic KS system of consumption type, without vacuum, only smallness of $ \norm{\nabla \rho_0}{L^4} $ in high (two or three) dimensions, and no smallness in the one-dimensional case, is required to guarantee the existence of the global smooth solution. Moreover, the global smooth solution can be obtained as the large friction limit of the almost pressureless EP system. In addition, the asymptotic behavior of the one-dimensional flow with vacuum is investigated in section \ref{sec:1-d-discussion}, below.

    \subsubsection*{Methodology and remarks}

    Instead of using the relative entropy method as in \cite{gallenmullerCahnHillardKeller2024}, where it is required that the solution to the limit KS system is more regular than that to the primitive EP system, we use a direct singular limit method. To be more precise, we introduce an intermediate system of system \eqref{sys:PLEP-rescaled}, where the feature/structure of the KS system is explicit; see \eqref{sys:PLEP-perturbation}, below. At the same time, the singular structure of the system will be separated from the limit system as $ \varepsilon \rightarrow 0 $ and can be controlled in the singular limit regime. The benefit of this method is that the regularity of the solutions to the singular EP system and the limit KS system are exactly the same. It is then clear that, in this case, no shock or any singularity will form in the EP system as $ \varepsilon \rightarrow 0 $. 

    Our initial data is required to be somewhat prepared to have naturally bounded energy in the singular limit regime; see section \ref{sec:l2-energy}. However, this is different from the well-prepared initial data as in the study of the small Mach number limit \cite{Klainerman1982}. As indicated below by the spectrum analysis in section \ref{sec:spectrum}, our initial data is ``large'' in term of order. Moreover, the time derivative is in general not uniformly bounded with respect to $ \varepsilon $; see \eqref{ene:802}, below. This is due to the strong initial layer as $ \varepsilon \rightarrow 0 $, when the damping effect damps out the error between the solutions of the EP system and the KS system within a very small initial time interval. In particular, we do not require the initial energy to be small; see \eqref{initial:000}, below. Therefore we cannot obtain a convergence rate estimate without better prepared data. 

    Our result holds also for the attracting EP system and the KS system of production type, for non-vacuum data and for finite time. Namely, the solutions to the attracting EP system will converge to the solution of the KS system of production type as $ \varepsilon \rightarrow 0 $ in the same setting but the convergence holds only for finite time. This is the best result one can obtain for the high friction limit since both the attracting EP system and the KS system of production type are expected to blow up in finite time with general initial data. 

    Lastly, we remark that the non-vacuum condition, i.e., \eqref{initial:001}, below, is essential for our result. In fact, when vacuum appears, it has been shown in \cite{jeongWellPosednessSingularityFormation2022} that the local well-posedness theory for the KS system is nontrivial and does not fit into the current framework. This can be also seen from the one-dimensional flow as discussed in section \ref{sec:1-d-discussion}. Indeed, any initial vacuum interval will shrink to a vacuum point as $ t \rightarrow \infty $ in the KS system of consumption type,  while the solution approaches the global equilibrium $ \sigma_{e,2} = M $ almost everywhere. Therefore, the (non-weighted) $H^s$-norm, $s\geq 1$, of the solution will grow to infinity as $ t \rightarrow \infty $. Such a behavior of the solution does not fit into our framework for the large friction limit and is left as a future work.

\subsection{Reformulation of system \eqref{sys:PLEP-limit}}
\label{sec:reformulation}

Motivated by \eqref{eq:rsPLEPlt-limit}, we define the Keller-Segel map for system \eqref{sys:PLEP-rescaled} as follows:
\begin{equation}
    \label{def:KS-map}
    \vec v_\rho: \quad \rho \mapsto  - \nabla \phi = - \nabla (-\Delta)^{-1} (\rho -M),
\end{equation}
satisfying 
\begin{equation}
    \label{def:KS-map-2}
     \vec v_\rho + \nabla \phi = 0.
\end{equation}
Furthermore, let 
\begin{equation}
    \label{def:perturbation-u}
    \vec u := \varepsilon \vec v_\rho + \varepsilon^{\alpha} \vec w.
\end{equation}
Then $ (\rho, \vec w) $ satisfies the following system:
\begin{subequations}
    \label{sys:PLEP-perturbation}
    \begin{align}
        \label{eq:ptb-01}
        \dtau \rho + \frac{1}{\varepsilon^{1-\alpha}} \dv(\rho \vec w) + \dv (\rho \vec v_\rho) & = 0, & \text{in} \ \Omega \times (0,\infty), \\
        \label{eq:ptb-02}
        \varepsilon^\alpha \rho \dtau \vec w + \varepsilon^{\alpha -1} \rho \vec u \cdot \nabla \vec w + \frac{1}{\varepsilon^{1-\alpha}}\nabla p(\rho) + \frac{1}{\varepsilon^{2-\alpha}} \rho \vec w & = -  \varepsilon \rho \dtau \vec v_\rho - \rho \vec u \cdot \nabla \vec v_\rho, & \text{in} \ \Omega \times (0,\infty).
    \end{align}
\end{subequations}

\begin{remark}
    Our ansatz \eqref{def:perturbation-u} is not an indication of $ \vec w = \mathcal O(1) $, but is designed to write system \eqref{sys:PLEP-perturbation} in a symmetric form, which simplifies our presentation below. 
\end{remark}

\subsection{Main results}
\label{sec:main-results}
We consider initial data $ (\rho_0, \vec w_0 ) $ to system \eqref{sys:PLEP-perturbation}, satisfying 
\begin{equation}
    \label{initial:000}
    \norm{\varepsilon^{\alpha/2} \vec w_0}{H^3} + \norm{\rho_0-M}{H^3} < \infty. 
\end{equation}
Moreover, let $ \underline\rho, \overline\rho$ be constants such that
\begin{equation}
    \label{initial:001}
    0 < \underline\rho < \rho_0 < \overline \rho,  \qquad \underline\rho < M < \overline \rho, \qquad \forall \vec x \in \Omega. 
\end{equation}
Here $ H^3 = H^3(\Omega) $ is the standard Sobolev space and $ \Omega = \mathbb R^3 \, \text{or} \, \mathbb T^3 $.

In addition,  $ \rho_0 $ satisfies
\begin{equation}
    \label{initial:002}
    \int_{\Omega} (\rho_0 - M) \idx = 0.
\end{equation}

We work on the classical solution to system \eqref{sys:PLEP-perturbation}, defined as follows:
\begin{definition}[Classical solution] 
    For any $ T \in (0,\infty) $, 
    we say $ (\rho, \vec w ) $ is a classical solution in the time interval $ (0,T) $ to system \eqref{sys:PLEP-perturbation} if 
    \begin{equation}
        \label{def:classical-sol}
        \rho, \, \vec w \in L^\infty(0,T;H^3(\Omega)) \qquad \text{and} \qquad \dtau \rho, \, \dtau \vec w \in L^\infty(0,T;H^2(\Omega)),
    \end{equation}
    and $ (\rho, \vec w) $ satisfies system \eqref{sys:PLEP-perturbation}. 
\end{definition}

Our first main theorem is concerning the uniform-in-$\varepsilon$ regularity of the classical solution, as follows: 

\begin{theorem}[Uniform regularity]
    \label{thm:uniform}
    Let $ 0 < \alpha < 2 $.
    Consider system \eqref{sys:PLEP-perturbation} equipped with initial data $ (\rho_0, \vec w_0) $ satisfying \eqref{initial:000}--\eqref{initial:002}. 
    \begin{enumerate}[label = (\roman*), ref= \ref{thm:uniform} (\roman*)]
        \item \label{thm:uniform-1} There exist $ T \in (0,\infty) $ and small enough $ \varepsilon_1 \in (0,1) $ such that, for all $ \varepsilon \in (0,\varepsilon_1) $, the unique classical solution with initial data $ (\rho_0, \vec w_0) $ to system \eqref{sys:PLEP-perturbation} exists for all $ \tau \in (0,T) $. Furthermore, the solution $ (\rho, \vec w) $ satisfies
        \begin{equation}
            \label{est:thm-01}
            \sup_{0\leq \tau \leq T} \bigl( \norm{\varepsilon^{\alpha/2} \vec w(\tau)}{H^3}^2 + \norm{\rho(\tau)-M}{H^3}^2 \bigr) + \int_0^T \bigl( \norm{\varepsilon^{\alpha/2-1}\vec w(s)}{H^3}^2 + \norm{\rho(s)-M}{H^3}^2 \bigr) \,ds \leq \mathcal C_1 < \infty,  
        \end{equation}
        for some $ \mathcal C_1 \in (0,\infty) $ depending only on the initial data. 

        \item \label{thm:uniform-2} If in addition, $ \norm{\nabla \rho_0}{L^4} $ is small enough, there exist $ \mathfrak c_0 $ and small enough $ \varepsilon_2 \in (0,1) $ such that for all $ \varepsilon \in (0,\varepsilon_2) $, the unique classical solution with initial data $ (\rho_0, \vec w_0) $ to system \eqref{sys:PLEP-perturbation} exists for all $ \tau \in (0,\infty) $. Furthermore, the solution $ (\rho, \vec w) $ satisfies
        \begin{equation}
            \label{est:thm-02}
            \begin{aligned}
                & \sup_{0\leq \tau < \infty}  e^{\mathfrak c_0 \tau}  \bigl( \norm{\varepsilon^{\alpha/2} \vec w(\tau)}{H^3}^2 + \norm{\rho(\tau)-M}{H^3}^2  \bigr)  \\
                & \qquad + \int_0^\infty e^{\mathfrak c_0 \tau} \bigl( \norm{\varepsilon^{\alpha/2-1} \vec w(s)}{H^3}^2 + \norm{\rho(s)-M}{H^3}^2  \bigr) \,ds \leq \mathcal C_2 < \infty,
            \end{aligned}
        \end{equation}
        for some $ \mathcal C_2 \in (0,\infty) $ depending only on the initial data. Moreover, there exist $ \mathfrak c_1 \in (0,\infty) $ and $ \mathfrak c_2 \in (0,\infty)$ such that
        \begin{equation}
            \sup_{0\leq \tau < \infty} \bigl( e^{\mathfrak c_1 \tau} \norm{\rho(\tau)-M}{L^\infty} + e^{\mathfrak c_2 \tau} \norm{\nabla\rho(\tau)}{L^4} \bigr) \leq \mathcal C_3 < \infty, 
        \end{equation}
        for some $ \mathcal C_3 \in (0,\infty) $ depending only on the initial data. 
    \end{enumerate}

\end{theorem}

\begin{remark}
    We emphasize that no uniform-in-$ \varepsilon $ estimate is available for the time derivative; see \eqref{ene:802}. This is due to the singular limit problem we are considering, and is a indicator of the initial layer for our solutions. 
\end{remark}

\begin{proof}
    Theorem \ref{thm:uniform-1} is the consequence of the estimates in sections \ref{sec:l2-energy}, \ref{sec:h3-energy}, \ref{sec:local-estimate}, and \ref{sec:closing}. Theorem \ref{thm:uniform-2} is the consequence of the estimates in sections \ref{sec:l2-energy}, \ref{sec:h3-energy}, \ref{sec:decay-l-infty-4}, \ref{sec:global-estimate}, and \ref{sec:closing}. 
\end{proof}

Our second main theorem is concerning the asymptotic limit, i.e. large friction limit, as follows: 

\begin{theorem}[Large friction limit]
    \label{thm:limit}
    Under the same assumptions as in theorem \ref{thm:uniform}, there exists a $ \sigma \in L^\infty(0,T;H^2)\cap L^2(0,T;H^3) $ with $ \dtau \sigma \in L^2(0,T;H^2) $, such that as $ \varepsilon \rightarrow 0 $, one has that 
\begin{align}
    \label{thm:limit-1}
    \rho - M  \stackrel{*}{\rightharpoonup} & ~ \sigma - M && \text{in} ~ L^\infty(0,T;H^3);\\
    \label{thm:limit-2}
    \rho - M {\rightharpoonup} & ~ \sigma - M && \text{in} ~ L^2(0,T;H^3); \\
    \label{thm:limit-3}
    \rho - M \rightarrow & ~ \sigma - M && \text{in} ~ C(0,T;H^2_\mrm{loc}); \\
    \label{thm:limit-4}
    \dtau \rho \rightharpoonup & ~ \dtau \sigma && \text{in} ~ L^2(0,T;H^2).
\end{align}
Here $ T \in (0,\infty] $ is the existence time of $ (\rho,\vec w) $ from theorem \ref{thm:uniform}. Moreover, $ \sigma $ solves the limiting system \eqref{sys:PLEP-limit}, and $ \sigma(\tau = 0) = \rho_0 $. 
    
\end{theorem}

\begin{proof}
    The proof of theorem \ref{thm:limit} is the consequence of section \ref{sec:convergence}.
\end{proof}

As a corollary, we also have the following result for the one-dimensional flow:
\begin{corollary}
    \label{cor:1-d}
    The results of theorem \ref{thm:uniform-2} and theorem \ref{thm:limit} hold for the one dimensional flow with $ H^3 $ replaced by $ H^2 $ and $ \norm{\nabla \rho_0}{L^4} = \mathcal O(1) $, i.e., without the smallness assumption on $ \norm{\nabla\rho_0}{L^4} $. 
\end{corollary}

\begin{proof}
    The corresponding convergence result in the one space dimension follows the same arguments as in section \ref{sec:convergence} once the uniform-in-$\varepsilon$ estimate is established. The latter can be established following the proof of theorem \ref{thm:uniform} with minimal modification; see section \ref{sec:1-d-uniform-est} for more detailed discussion. 
\end{proof}

\subsection{Spectrum analysis}
\label{sec:spectrum}
In this section, we would like to analyze the spectrum of the linear operator associated with system \eqref{sys:PLEP-rescaled}. This provides, from the perspective of spectral theory, the rigidity for the ansatz of $ \vec u $ in \eqref{def:perturbation-u} and the high friction limit $ \varepsilon\rightarrow 0 $. 

We first linearize system \eqref{sys:PLEP-rescaled} around the equilibrium $ (\rho_e, \vec u_e) = (M,0) $. Then one has that
\begin{subequations}
    \label{sys:linear}
    \begin{align}
        \label{leq:cnt}
        \dtau \delta \rho + \frac{1}{\varepsilon} M \dv \delta\vec u & =  0, \\
        \label{leq:mmt}
        M \dtau \delta \vec u + \frac{\gamma}{\varepsilon^{1-\alpha}} M^{\gamma-1} \nabla \delta\rho + \frac{1}{\varepsilon^2} M \delta \vec u + \frac{1}{\varepsilon} M \nabla \delta \phi & = 0, \\
        \label{leq:psn}
        - \Delta \delta \phi & = \delta \rho,
    \end{align}
    where $ (\delta \rho, \delta \vec u) $ are the linearized variables. 
\end{subequations}
Then we look for solutions to \eqref{sys:linear} of the form 
\begin{equation}
    \label{lene:001}
    (\delta \rho (\vec x, \tau), \delta \vec u (\vec x, \tau)) = e^{\lambda\tau} (\zeta(\vec x), \vec U(\vec x)),
\end{equation}
with
\begin{subequations}\label{sys:lene-1}
\begin{align}
\label{lene:002}
\lambda \zeta + \frac{1}{\varepsilon} M \dv \vec U & = 0, \\
\label{lene:003}
\lambda M \vec U + \frac{\gamma}{\varepsilon^{1-\alpha}} M^{\gamma-1} \nabla \zeta + \frac{1}{\varepsilon^2} M \vec U + \frac{1}{\varepsilon} M \nabla (-\Delta)^{-1} \zeta & = 0. 
\end{align}
\end{subequations}
Then one can derive from system \eqref{sys:lene-1} that 
\begin{equation}
    \label{lene:004}
    - ( \varepsilon^2 \lambda^2 + \lambda + M) \zeta = \gamma M^{\gamma-1} \varepsilon^\alpha (-\Delta) \zeta,
\end{equation}
and
\begin{equation}
    \label{lene:005}
    \vec U = - \frac{1}{(1 + \varepsilon^2 \lambda) M} \bigl\lbrack \varepsilon M \nabla (-\Delta)^{-1} \zeta + \varepsilon^{\alpha + 1} \gamma M^{-\gamma-1} \nabla\zeta \bigr\rbrack. 
\end{equation}

Then one can see from \eqref{lene:004} and \eqref{lene:005} that in the asymptotic regime $ \varepsilon \rightarrow 0 $, $ \lambda < 0 $ and therefore the equilibrium is linearly stable. Moreover, for any fixed $ \lambda $, one has that
\begin{equation}
    \label{lene:006}
    \zeta = \mathcal O(1) \qquad \text{and} \qquad \vec U = \mathcal O(\varepsilon) + \mathcal O(\varepsilon^{\alpha+1}).
\end{equation}
Consequently, it is reasonable to assume 
\begin{equation}
    \label{lene:007}
    \delta \rho = \mathcal O(1) \qquad \text{and} \qquad \delta \vec u = \mathcal O(\varepsilon) + \mathcal O(\varepsilon^{\alpha+1}).
\end{equation}
However, our assumption of $ \rho_0 $ and $ \vec w_0 $ in \eqref{initial:000} and the ansatz \eqref{def:perturbation-u} implies that
\begin{equation}
    \label{lene:008}
    \rho\vert_{\tau=0} = \mathcal O(1) \qquad \text{and} \qquad \vec u\vert_{\tau=0} = \mathcal O(\varepsilon) + \mathcal O(\varepsilon^{\alpha/2}), 
\end{equation}
which is weaker than \eqref{lene:007}.

\section{Uniform-in-$\varepsilon$ estimates}
\label{sec:uniform-estimates}

The goal of this section to derive uniform-in-$\varepsilon $ estimates for the solution of system \eqref{sys:PLEP-perturbation}, in particular, the uniform bounds for the density $ \rho $. To simplify the presentation, we make the {\it a priori} assumption that $ \rho $ is uniformly bounded from above, i.e., 
\begin{equation}
    \label{ass-prior:rho}
    0 < \frac{1}{2} \underline{\rho} \leq \rho(x,t) \leq 2 \overline \rho  < \infty. 
\end{equation}
We will verify this assumption in section \ref{sec:closing}. Moreover, all calculation below will be done in $ \mathbb R^3 $, while the calculation in $ \mathbb T^3 $ will be similar with minimal modification and thus is omitted. 

\subsection{The Keller-Segel map}

Thanks to \eqref{def:KS-map}, one has that  $ - \Delta \vec v_\rho = - \nabla(\rho - M) $. Hence, the elliptic regularity theory yields that
for any $ s $ and $ 1 < p < \infty $, we have
\begin{gather}
    \label{est:KS-map}
    \norm{\vec v_\rho}{W^{s+2,p}} \leq C_s \norm{\nabla (\rho - M)}{W^{s,p}}, \\
    \intertext{and, in particular} 
    \label{est:KS-map-2}
    \norm{\nabla \vec v_\rho}{L^\infty} \lesssim \norm{\vec v_\rho}{W^{2,4}} \lesssim \norm{\nabla (\rho - M)}{L^{4}} \lesssim \norm{\rho - M}{H^2}.
\end{gather}

\subsection{The $ L^2 $-estimate}\label{sec:l2-energy}
Taking the $ L^2 $-inner product of \eqref{eq:ptb-02} with $ \vec w $ leads to the following:
\begin{equation}
    \label{est:001}
    \begin{gathered}
        \dfrac{d}{d\tau} \biggl\lbrace \frac{\varepsilon^\alpha}{2} \int \rho \vert \vec w \vert^2 \idx  
        + \frac{1}{\gamma-1} \int \bigl(\rho^\gamma - M^\gamma - \gamma M^{\gamma-1}(\rho - M) \bigr)  \idx \biggr\rbrace + \frac{1}{\varepsilon^{2-\alpha}} \int \rho \vert \vec w \vert^2 \idx \\ =   \int \nabla \rho^\gamma \cdot \vec v_\rho  \idx 
        - \varepsilon \int \rho \dtau \vec v_\rho \cdot \vec w \idx  -\int \rho \vec u \cdot \nabla \vec v_\rho \cdot \vec w \idx =: \sum_{j = 1}^3 I_j,
    \end{gathered}
\end{equation}
where we have used \eqref{eq:ptb-01} to calculate
\begin{equation}
    \label{est:002}
    \begin{gathered}
    \frac{1}{\varepsilon^{1-\alpha}} \int \nabla p(\rho) \cdot \vec w \idx = \frac{1}{\varepsilon^{1-\alpha}} \frac{\gamma}{\gamma-1} \int \nabla \rho^{\gamma-1} \cdot \rho \vec w \idx = \frac{\gamma}{\gamma-1} \int \rho^{\gamma-1} (\dtau \rho + \dv(\rho \vec v_\rho)) \idx \\
    = \frac{1}{\gamma-1} \int \dtau \rho^\gamma \idx + \int \rho^\gamma \dv \vec v_\rho \idx = \dfrac{d}{d\tau} \frac{1}{\gamma-1} \int  \bigl(\rho^\gamma - M^\gamma - \gamma M^{\gamma-1}(\rho - M) \bigr)   \idx \\
    - \int \nabla \rho^\gamma \cdot \vec v_\rho \idx.
\end{gathered}
\end{equation}

\smallskip 

Next, we estimate the terms $ I_j $ for $ j = 1,2,3 $. Thanks to \eqref{def:KS-map} and \eqref{def:perturbation-u}, we have 
\begin{equation}
    \label{est:001-1}
    \begin{gathered}
        \vert I_3 \vert = \vert - \int \rho (\varepsilon \vec v_\rho + \varepsilon^\alpha \vec w) \nabla \vec v_\rho \cdot \vec w \idx \vert \lesssim \varepsilon \overline\rho^{1/2} \norm{\vec v_\rho}{L^2} \norm{\nabla \vec v_\rho}{L^\infty} \norm{\rho^{1/2} \vec w}{L^2} \\
         + \varepsilon^\alpha \norm{\nabla \vec v_\rho}{L^\infty} \norm{\rho^{1/2} \vec w}{L^2}^2 \lesssim \varepsilon \overline{\rho}^{1/2} \norm{\rho-M}{L^2} \norm{\rho-M}{H^2} \norm{\rho^{1/2} \vec w}{L^2} \\
         + \varepsilon^\alpha \norm{\rho-M}{H^2} \norm{\rho^{1/2} \vec w}{L^2}^2 \\
         \lesssim \varepsilon^{2} \norm{\rho - M}{H^2}  \norm{\varepsilon^{\alpha/2-1} \rho^{1/2} \vec w}{L^2}^2 + \varepsilon^{2-\alpha} \overline \rho  \norm{\rho -M}{H^2}\norm{\rho-M}{L^2}^2,
    \end{gathered}
\end{equation}
where we have used \eqref{est:KS-map} and \eqref{est:KS-map-2} and applied Cauchy-Schwartz inequality. 

To estimate $ I_2 $, we notice that, thanks to \eqref{def:KS-map} and \eqref{eq:ptb-01}, we have 
\begin{equation}
    \label{def:dtau-v-rho}
    \dtau \vec v_\rho  = \nabla (-\Delta)^{-1}( \varepsilon^{\alpha - 1} \dv (\rho \vec w) + \dv (\rho \vec v_\rho) ),
\end{equation}
and hence,
\begin{equation}\label{est:001-3}
    \norm{\dtau \vec v_\rho}{L^2} \lesssim \varepsilon^{\alpha - 1} \overline\rho^{1/2} \norm{\rho^{1/2} \vec w}{L^2} + \overline\rho\norm{\rho - M}{L^2}.
\end{equation}
Therefore, we have
\begin{equation}
    \label{est:001-4}
    \begin{gathered}
        \vert I_2 \vert \lesssim \varepsilon \overline\rho^{1/2} \norm{\dtau \vec v_\rho}{L^2} \norm{\rho^{1/2}\vec w}{L^2} \lesssim \varepsilon^{2} \overline\rho  \norm{\varepsilon^{\alpha/2-1} \rho^{1/2} \vec w}{L^2}^2 + \varepsilon^{2-\alpha} \overline\rho^2 \norm{\rho - M}{L^2}^2. 
    \end{gathered}
\end{equation}

To estimate $ I_1 $, we apply integration by parts as follows: 
\begin{equation}
    \label{est:001-5} 
\begin{gathered}
    I_1 + \gamma (\underline\rho/2)^{\gamma-1} \norm{\rho-M}{L^2}^2 \leq I_1 + \gamma \int \rho^{\gamma-1} (\rho - M)^2\idx \\
    = \int \gamma  \rho^{\gamma-1} \nabla (\rho - M) \cdot \vec v_\rho \idx + \gamma \int \rho^{\gamma-1} (\rho - M)^2\idx = - \int \gamma (\rho - M) \nabla \rho^{\gamma-1} \cdot \vec v_\rho \idx  \\
    = - \gamma(\gamma-1) \int \rho^{\gamma-2} (\rho-M) \nabla (\rho - M) \cdot \vec v_\rho \idx  \\
    \lesssim_{\underline\rho, \overline\rho} 
    \norm{\rho-M}{L^\infty} \norm{\nabla(\rho-M)}{H^{-1}} \norm{\vec v_\rho}{H^1} \\
    \lesssim \norm{\rho-M}{L^\infty} \norm{\rho-M}{L^2}^2.
\end{gathered}
\end{equation}

\smallskip

To conclude this section, let $ \mathcal E_0 $ be the energy functional defined by
\begin{equation}
    \label{def:e-0}
    \mathcal E_0(t):= \frac{\varepsilon^\alpha}{2} \int \rho \vert \vec w \vert^2 \idx + \frac{1}{\gamma-1} \int \lbrack \rho^\gamma - M^\gamma - \gamma M^{\gamma-1} (\rho - M) \rbrack \idx.
\end{equation}
Thanks to \eqref{ass-prior:rho}, one has that
\begin{equation}
    \label{est:003}
    \mathcal E_0 {\gtrsim}_{\underline \rho, \overline \rho}   \varepsilon^{\alpha} \norm{\vec w}{L^2}^2 + \norm{\rho - M}{L^2}^2. 
\end{equation}
Meanwhile, let $ \mathcal D_0 $ be the dispersion functional defined by
\begin{equation}
    \label{def:d-0}
    \mathcal D_0(t) := \varepsilon^{\alpha-2} \norm{\vec w}{L^2}^2 + \norm{\rho - M}{L^2}^2 \gtrsim_{\underline \rho, \overline \rho} \mathcal E_0.
\end{equation}

Then from \eqref{est:001}--\eqref{est:001-5}, one can conclude that
\begin{equation}
    \label{ene:004}
    \dfrac{d}{d\tau} \mathcal  E_0(t) + \mathcal D_0(t) \lesssim_{\underline\rho, \overline\rho}  (\varepsilon^2 + \varepsilon^{2-\alpha}) (\norm{\rho - M}{H^2} + 1) \mathcal D_0(t) + \norm{\rho-M}{L^\infty} \mathcal E_0(t). 
\end{equation}

\subsection{The $ H^3 $-estimate}
\label{sec:h3-energy}

In the following, we will abuse the notation and treat $ \partial $ as a generic spatial derivative, regardless of the direction, i.e., $ \partial \in \lbrace \partial_x, \partial_y, \partial_z \rbrace $. Additionally, we omit the combinatorial constants from the Leibniz's rule in the following. This will simplify the presentation but not change the analysis. First, to write down the system of equations for $ (\partial^3 \rho, \partial^3 \vec w) $, from \eqref{sys:PLEP-perturbation}, one can easily obtain
\begin{subequations}
    \label{sys:PLEP-perturbation-ddd}
    \begin{gather}
        \label{eq:ptb-ddd-01}
        \begin{gathered}
            \dtau \partial^3 \rho + (\frac{1}{\varepsilon^{1-\alpha}}\vec w + \vec v_\rho) \cdot \nabla \partial^3 \rho + \frac{1}{\varepsilon^{1-\alpha}} \rho \dv \partial^3 \vec w  + \rho \partial^3 (\rho - M) = - \sum_{l = 1}^3 \partial^l \rho \partial^{3-l}(\rho - M) \\
            - \sum_{l=1}^{3} (\frac{1}{\varepsilon^{1-\alpha}} \partial^l \vec w + \partial^l \vec v_\rho) \cdot \nabla\partial^{3-l} \rho - \frac{1}{\varepsilon^{1-\alpha}} \sum_{l=1}^3 \partial^{l} \rho \dv \partial^{3-l}\vec w, 
        \end{gathered}\\
        \label{eq:ptb-ddd-02}
        \begin{gathered}
            \varepsilon^\alpha \rho \dtau \partial^3 \vec w + \varepsilon^{\alpha-1} \rho \vec u \cdot \nabla \partial^3 \vec w + \frac{\gamma}{\varepsilon^{1-\alpha}}  \rho^{\gamma-1} \nabla \partial^3 \rho + \frac{1}{\varepsilon^{2-\alpha}} \rho \partial^3 \vec w = - \varepsilon \rho \dtau \partial^3 \vec v_\rho - \rho \partial^3 (\vec u \cdot \nabla \vec v_\rho) \\
            - \sum_{l=1}^3 \lbrack \varepsilon^{\alpha-1} \rho \partial^l \vec u \cdot \nabla \partial^{3-l} \vec w + \frac{\gamma}{\varepsilon^{1-\alpha}} \rho \partial^l \rho^{\gamma-2} \nabla \partial^{3-l}\rho \rbrack. 
        \end{gathered}
    \end{gather}

\end{subequations}

Taking the $ L^2 $-inner product of \eqref{eq:ptb-ddd-02} with $ \partial^3 \vec w $ leads to the following:
\begin{equation}
    \label{est:101}
    \begin{gathered}
        \dfrac{d}{d\tau} \biggl\lbrace \frac{\varepsilon^\alpha}{2} \int \rho \vert \partial^3 \vec w \vert^2 \idx + \frac{\gamma}{2} \int \rho^{\gamma-2} \vert \partial^3 \rho \vert^2 \idx  \biggr\rbrace + \frac{1}{\varepsilon^{2-\alpha}} \int \rho \vert \partial^3 \vec w \vert^2 \idx + \gamma \int \rho^{\gamma-1}\vert \partial^3 \rho \vert^2 \idx \\
        = \sum_{j=4}^6 I_j -  \varepsilon \int \rho \dtau \partial^3 \vec v_\rho \cdot \partial^3 \vec w \idx - \int \rho \partial^3(\vec u \cdot \nabla\vec v_\rho ) \cdot \partial^3 \vec w \idx \\
        - \sum_{l=1}^3 \int \varepsilon^{\alpha-1} \rho \partial^l \vec u \cdot \nabla \partial^{3-l} \vec w \cdot \partial^3 \vec w \idx 
        - \sum_{l=1}^3 \int \frac{\gamma}{\varepsilon^{1-\alpha}} \rho \partial^l \rho^{\gamma-2} \nabla \partial^{3-l}\rho \cdot \partial^3 \vec w \idx\\
        =: \sum_{j=4}^{10} I_j,
    \end{gathered}
\end{equation}
where we have used \eqref{eq:ptb-ddd-01} to calculate 
\begin{equation}
    \label{est:102}
    \begin{gathered}
        \int \frac{\gamma}{\varepsilon^{1-\alpha}}  \rho^{\gamma-1} \nabla \partial^3 \rho \cdot \partial^3 \vec w \idx = - \int \frac{\gamma}{\varepsilon^{1-\alpha }} \partial^3 \rho \nabla \rho^{\gamma-1} \cdot \partial^3 \vec w \idx - \int \frac{\gamma}{\varepsilon^{1-\alpha}} \rho^{\gamma-1} \partial^3 \rho  \dv \partial^3 \vec w \idx \\
        \stackrel{\eqref{eq:ptb-ddd-01}}{=} - \int \frac{\gamma}{\varepsilon^{1-\alpha }} \partial^3 \rho \nabla \rho^{\gamma-1} \cdot \partial^3 \vec w \idx + \int \gamma \rho^{\gamma-2} \partial^3 \rho \biggl\lbrack \dtau \partial^3 \rho + (\frac{1}{\varepsilon^{1-\alpha}}\vec w + \vec v_\rho) \cdot\nabla \partial^3 \rho  + \rho \partial^3(\rho - M) \\
        + \sum_{l=1}^3 \partial^l \rho \partial^{3-l}(\rho - M) + \sum_{l=1}^3 (\frac{1}{\varepsilon^{1-\alpha}} \partial^l \vec w + \partial^l \vec v_\rho) \cdot \nabla \partial^{3-l} \rho + \frac{1}{\varepsilon^{1-\alpha}}\sum_{l=1}^3 \partial^l \rho \dv \partial^{3-l}\vec w \biggr\rbrack \idx \\
        = \dfrac{d}{d\tau} \frac{\gamma}{2} \int \rho^{\gamma-2} \vert \partial^3 \rho \vert^2 \idx  + \int \gamma \rho^{\gamma-1} \vert \partial^3 \rho \vert^2 \idx - \int \frac{\gamma}{\varepsilon^{1-\alpha }} \partial^3 \rho \nabla \rho^{\gamma-1} \cdot \partial^3 \vec w \idx  \\
        - \frac{\gamma}{2} \int \underbrace{\lbrace \dtau \rho^{\gamma-2} + \dv \lbrack \rho^{\gamma-2} (\frac{1}{\varepsilon^{1-\alpha}} \vec w + \vec v_\rho)\rbrack \rbrace}_{\stackrel{\eqref{eq:ptb-01}}{=} (3-\gamma) \rho^{\gamma-2}  (\frac{1}{\varepsilon^{1-\alpha}}\dv \vec w + \rho - M)} \vert \partial^3 \rho \vert^2 \idx \\
        + \int \gamma\rho^{\gamma-2} \partial^3 \rho \biggl\lbrack \sum_{l=1}^3 \partial^l \rho \partial^{3-l}(\rho - M) + \sum_{l=1}^3 (\frac{1}{\varepsilon^{1-\alpha}} \partial^l \vec w + \partial^l \vec v_\rho) \cdot \nabla \partial^{3-l} \rho \\
        + \frac{1}{\varepsilon^{1-\alpha}}\sum_{l=1}^3 \partial^l \rho \dv \partial^{3-l}\vec w \biggr\rbrack \idx =:\dfrac{d}{d\tau} \frac{\gamma}{2} \int \rho^{\gamma-2} \vert \partial^3 \rho \vert^2 \idx  + \int \gamma \rho^{\gamma-1} \vert \partial^3 \rho \vert^2 \idx - \sum_{j=4}^{6} I_{j}.
    \end{gathered}
\end{equation}

Next, we estimate the terms $ I_j $ for $ j = 4, \cdots, 10 $. 
Applying H\"older's inequality, the Sobolev embedding inequality, and Young's inequality yields
\begin{align}
    \label{est:101-01}
    & \begin{aligned}
        \vert I_4 \vert \lesssim_{\underline\rho, \overline\rho} & \frac{1}{\varepsilon^{1-\alpha}} \norm{\partial^3 \rho}{L^2} \norm{\partial^3 \vec w}{L^2} \norm{\nabla \rho}{L^\infty} \lesssim \varepsilon^{\alpha/2} \norm{\rho - M}{H^3} \norm{\partial^3\rho}{L^2} \norm{\varepsilon^{\alpha/2-1} \partial^3 \vec w}{L^2}, 
    \end{aligned} \\
    \label{est:101-02}
    & \begin{aligned}
        \vert I_5 \vert \lesssim_{\overline\rho} & \norm{\partial^3 \rho}{L^2}^2 \norm{ \frac{1}{\varepsilon^{1-\alpha}}\dv \vec w + \rho - M}{L^\infty} \lesssim \varepsilon^{\alpha/2} \norm{\rho - M}{H^3} \norm{\partial^3 \rho}{L^2}\norm{\varepsilon^{\alpha/2-1} \vec w}{H^3} \\
        & + \norm{\rho-M}{L^\infty} \norm{\partial^3\rho}{L^2}^2,
    \end{aligned}\\
    \label{est:101-03}
    & \begin{aligned}
        \vert I_7 \vert \lesssim_{\overline \rho} & \varepsilon \norm{\dtau\partial^3\vec v_\rho}{L^2} \norm{\partial^3 \vec w}{L^2} \stackrel{\eqref{def:dtau-v-rho}}\lesssim \varepsilon (\varepsilon^{\alpha -1} \norm{\rho \vec w}{H^3} + \norm{\rho \vec v_\rho}{H^3}) \norm{\partial^3 \vec w}{L^2} \\
        \stackrel{\eqref{est:KS-map}}{\lesssim_{\overline\rho}} & \varepsilon^2 (\norm{\rho-M}{H^3} + 1) \norm{\varepsilon^{\alpha/2-1}\vec w}{H^3}^2 + \varepsilon^{2-\alpha/2} (\norm{\rho - M}{H^3}^2 + 1) \norm{\varepsilon^{\alpha/2-1}\vec w}{H^3},
    \end{aligned}\\
    \label{est:101-04}
    & \begin{aligned}
        \vert I_8 \vert \lesssim_{\overline\rho} & \norm{\vec u}{H^3}\norm{\nabla\vec v_\rho}{H^3} \norm{\partial^3 \vec w}{L^2} \stackrel{\eqref{def:perturbation-u}}{\lesssim}  (\varepsilon \norm{\vec v_\rho}{H^3} + \varepsilon^\alpha\norm{\vec w}{H^3} ) \norm{\nabla\vec v_\rho}{H^3} \norm{\partial^3 \vec w}{L^2} \\
        \stackrel{\eqref{est:KS-map}}{\lesssim} & \varepsilon^{2-\alpha/2} \norm{\rho-M}{H^3}^2 \norm{\varepsilon^{\alpha/2-1}\vec w}{H^3} + \varepsilon^2 \norm{\rho-M}{H^3} \norm{\varepsilon^{\alpha/2-1}\vec w}{H^3}^2,
    \end{aligned}\\
    \label{est:101-05}
    & \begin{aligned}
        \vert I_9 \vert \lesssim_{\overline\rho} & \varepsilon^{\alpha-1}\norm{\vec u}{H^3} \norm{\vec w}{H^3} \norm{\partial^3 \vec w}{L^2} \stackrel{\eqref{def:perturbation-u}}{\lesssim} \varepsilon^{\alpha-1}(\varepsilon \norm{\vec v_\rho}{H^3} + \varepsilon^\alpha\norm{\vec w}{H^3} ) \norm{\vec w}{H^3} \norm{\partial^3 \vec w}{L^2} \\
        \stackrel{\eqref{est:KS-map}}{\lesssim} &  \varepsilon^{2} \norm{\rho-M}{H^2} \norm{\alpha^{\alpha/2 -1}\vec w}{H^3}^2 + \varepsilon^{1+\alpha/2} \norm{\varepsilon^{\alpha/2}\vec w}{H^3} \norm{\alpha^{\alpha/2 -1}\vec w}{H^3}^2,
    \end{aligned}\\
    \label{est:101-06}
    & \begin{aligned}
        \vert I_{10} \vert \lesssim_{\overline\rho} & \varepsilon^{\alpha/2} (\norm{\rho - M}{H^3}^2 + 1 ) \norm{\rho-M}{H^3}^2 \norm{\varepsilon^{\alpha/2-1}\partial^3 \vec w}{L^2}.
    \end{aligned}
\end{align}
Meanwhile, $ I_6 $ can be estimated by
\begin{equation}
    \label{est:101-07}
    \begin{aligned}
    \vert I_6 \vert \lesssim_{\underline\rho, \overline\rho} & \sum_{l=1}^3 \norm{\partial^3 \rho}{L^2} \norm{\partial^l \rho \partial^{3-l}\rho}{L^2} + \sum_{l=1}^3 \norm{\partial^3 \rho}{L^2} \norm{\partial^l \vec v_\rho \cdot \nabla \partial^{3-l}\rho}{L^2} \\
        & + \sum_{l=1}^3 \norm{\partial^3 \rho}{L^2} \norm{\frac{1}{\varepsilon^{1-\alpha}}\partial^l \vec w \cdot \nabla \partial^{3-l}\rho}{L^2} 
         + \sum_{l=1}^3 \norm{\partial^3 \rho}{L^2} \norm{ \frac{1}{\varepsilon^{1-\alpha}}\partial^l \rho \partial^{3-l} \vec w}{L^2} =: \sum_{j=1}^{4}I_{6,j}. 
    \end{aligned}
\end{equation}
One can estimate $ I_{6,3} $ and $ I_{6,4} $ similarly as before:
\begin{equation}
    \label{est:101-07-1}
    \vert I_{6,3} \vert + \vert I_{6,4} \vert \lesssim \varepsilon^{\alpha/2}\norm{\rho - M}{H^3}^2 \norm{\varepsilon^{\alpha/2-1}\vec w}{H^3}.
\end{equation}
The estimates of $ I_{6,1} $ is a bit more involved. Notice that the Gagliardo-Nirenberg inequality implies that
\begin{align}
    \label{est:101-07-2}
    \norm{\nabla\rho}{L^6} \lesssim \norm{\rho-M}{L^\infty}^{2/3} \norm{\nabla^3 \rho}{L^2}^{1/3}, \\
    \label{est:101-07-3}
    \norm{\nabla^2\rho}{L^3} \lesssim \norm{\rho-M}{L^\infty}^{1/3} \norm{\nabla^3 \rho}{L^2}^{2/3}. 
\end{align}
Therefore, applying H\"older's inequality, \eqref{est:101-07-2}, and \eqref{est:101-07-2} in $ I_{6,1} $ implies that
\begin{equation}\label{est:101-07-4}
    \vert I_{6,1} \vert \lesssim_{\overline\rho} \norm{\partial^3\rho}{L^2} \norm{\partial\rho}{L^6}\norm{\partial^2\rho}{L^3} + \norm{\rho-M}{L^\infty} \norm{\partial^3 \rho}{L^2}^2 \\
    \lesssim \norm{\rho-M}{L^\infty} \norm{\partial^3 \rho}{L^2}^2.
\end{equation} 
Similarly, $ I_{6,2} $ can be estimated as follows:
\begin{equation}
    \label{est:101-07-5}
    \begin{aligned}
        \vert I_{6,2} \vert \lesssim_{\overline\rho} & \norm{\partial^3 \rho}{L^2} \norm{\partial \vec v_\rho}{L^\infty} \norm{\nabla\partial^2 \rho}{L^2} + \norm{\partial^3 \rho}{L^2} \norm{\partial^2 \vec v_\rho}{L^6} \norm{\nabla\partial \rho}{L^3} \\
        & + \norm{\partial^3 \rho}{L^2} \norm{\partial^3 \vec v_\rho}{L^3} \norm{\nabla \rho}{L^6} \lesssim  \norm{\partial^3 \rho}{L^2} \norm{\partial \vec v_\rho}{W^{1,4}} \norm{\nabla\partial^2 \rho}{L^2} \\
        & + \norm{\partial^3 \rho}{L^2} \norm{\vec v_\rho}{W^{2,6}} \norm{\nabla\partial \rho}{L^3} 
         + \norm{\partial^3 \rho}{L^2} \norm{\vec v_\rho}{W^{3,3}} \norm{\nabla \rho}{L^6} \\
        \lesssim & \norm{\nabla \rho}{L^4} \norm{\rho-M}{H^3}^2  + \norm{\rho-M}{L^\infty} \norm{\rho - M}{H^3}^2,
    \end{aligned}
\end{equation}
where we have used the elliptic estimate \eqref{est:KS-map} and \eqref{est:101-07-2}--\eqref{est:101-07-3}.  

This finishes the estimate of $ (\partial^3 \rho, \partial^3 \vec w) $. The estimates of $ (\partial^j \rho, \partial^j \vec w) $, $ j =1,2 $, are similar and are omitted here. 

\begin{remark}
    In the one-dimensional case, one has that $ \norm{\partial v_\rho}{L^\infty} = \norm{\rho - M}{L^\infty} $, and hence there is no need for $ \norm{\nabla \rho}{L^4} $ for the estimate of $ I_{6,2} $. See section \ref{sec:1-d-uniform-est} for more details. 
\end{remark}

To conclude, let the high order energy functional be
\begin{equation}
    \label{def:e-1} 
    \mathcal E_1(\tau) := \sum_{j=1}^{3} \biggl\lbrace \frac{\varepsilon^\alpha}{2} \int \rho \vert \nabla^j \vec w \vert^2 \idx + \frac{\gamma}{2} \int \rho^{\gamma-2} \vert \nabla^j \rho \vert^2 \idx \biggr\rbrace,
\end{equation}
and the high order dispersion functional be
\begin{equation}
    \label{def:d-1}
    \mathcal D_1(\tau) := \sum_{j=1}^{3} \biggl\lbrace \frac{1}{\varepsilon^{2-\alpha}} \int \rho \vert \nabla^j \vec w \vert^2 \idx + \gamma \int \rho^{\gamma-1} \vert \nabla^j \rho \vert^2 \idx \biggr\rbrace. 
\end{equation}
Thanks to \eqref{ass-prior:rho}, we have, recalling \eqref{def:e-0} and \eqref{def:d-0}, that the total energy and dispersion functionals satisfy
\begin{align}
    \label{est:201}
    \mathcal E_\mrm{total}(\tau) := & \mathcal E_0(\tau) + \mathcal E_1(\tau) \gtrsim_{\underline \rho, \overline \rho} \norm{\varepsilon^{\alpha/2} \vec w}{H^3}^2 + \norm{\rho - M}{H^3}^2, \\
    \label{est:202}
    \mathcal D_\mrm{total}(\tau) := & \mathcal D_0(\tau) + \mathcal D_1(\tau) \gtrsim_{\underline\rho, \overline\rho} \norm{\varepsilon^{\alpha/2-1}\vec w}{H^3}^2 + \norm{\rho - M}{H^3}^2 \gtrsim_{\underline\rho,\overline\rho} \mathcal E_\mrm{total}(\tau).
\end{align}
Then, collecting \eqref{est:101}--\eqref{est:101-07-5} leads to 
\begin{equation}
    \label{ene:203}
    \begin{gathered}
        \dfrac{d}{d\tau} \mathcal E_1(\tau) + \mathcal D_1 (\tau) \lesssim_{\underline\rho, \overline\rho} (\varepsilon^{\alpha/2} + \varepsilon^{2-\alpha/2}) H(\mathcal E_\mrm{total}(\tau))\mathcal D_\mrm{total}(\tau) \\
        + (\norm{\rho-M}{L^\infty} + \norm{\nabla \rho}{L^4}) \mathcal E_\mrm{total}(\tau).
    \end{gathered}
\end{equation}
Hereafter $ H(\mathcal E_\mrm{total}) $ is a generic polynomial of $ \mathcal E_\mrm{total} $. 
Hence, \eqref{ene:004} and \eqref{ene:203} imply
\begin{equation}
    \label{ene:204}
    \begin{gathered}
        \dfrac{d}{d\tau}  \mathcal E_\mrm{total}(\tau) +  \mathcal D_\mrm{total}(\tau) \lesssim_{\underline \rho, \overline\rho} 
        (\varepsilon^{\alpha/2} + \varepsilon^{2-\alpha}) H(\mathcal E_\mrm{total}(\tau)) \mathcal D_\mrm{total}(\tau) \\
        + (\norm{\rho-M}{L^\infty} + \norm{\nabla \rho}{L^4}) \mathcal E_\mrm{total}(\tau).
    \end{gathered}
\end{equation}

\subsection{Local-in-time estimate}
\label{sec:local-estimate}
Thanks to \eqref{ene:204} and the Sobolev embedding inequality, one can conclude that 
\begin{equation}
    \label{ene:301}
    \dfrac{d}{d\tau}  \mathcal E_\mrm{total}(\tau) +  \mathcal D_\mrm{total}(\tau) \lesssim_{\underline \rho, \overline\rho} 
        (\varepsilon^{\alpha/2} + \varepsilon^{2-\alpha}) H(\mathcal E_\mrm{total}(\tau)) \mathcal D_\mrm{total}(\tau) + H(\mathcal E_\mrm{total}(\tau)). 
\end{equation}
Then for $ \varepsilon $ small enough, there exists $ T \in (0,\infty) $, depending on the initial data, such that
\begin{equation}
    \label{ene:302}
    \sup_{0\leq s  \leq T} \mathcal E_\mrm{total}(s) + \int_0^T \mathcal D_\mrm{total}(s) \,ds \leq 2 \mathcal E_\mrm{total}(0). 
\end{equation}

\subsection{Decay estimate of $ \norm{\rho-M}{L^\infty} $ and $ \norm{\nabla \rho}{L^4} $}
\label{sec:decay-l-infty-4}

Let $ \mathfrak{c}_0 > 0 $ be a constant to be determined later, and let the (time-)weighted energy functional be 
\begin{equation}
    \label{ene:400}
    \mathfrak F(\tau) := \sup_{0\leq s \leq \tau }e^{\mathfrak c_0 s} \mathcal E_\mrm{total}(s) + \int_0^\tau e^{\mathfrak c_0 s} \mathcal D_{\mrm{total}}(s)\,ds.
\end{equation}

\smallskip 

From \eqref{eq:ptb-01}, one can write down
\begin{equation}
    \label{ene:401}
    \dtau (\rho - M) + (\frac{1}{\varepsilon^{1-\alpha}}\vec w + \vec v_\rho) \cdot \nabla (\rho - M) + \rho (\rho - M) + \frac{1}{\varepsilon^{1-\alpha}} \rho \dv \vec w  = 0.
\end{equation}
Direct calculation yields that, for any $ p > 1 $,
\begin{equation}
    \label{ene:402}
    \begin{gathered}
    \dfrac{d}{d\tau} \norm{\rho-M}{L^p} + \underline\rho \norm{\rho-M}{L^p} \lesssim_{\underline\rho, \overline\rho} \frac{1}{\varepsilon^{1-\alpha}} \norm{\dv \vec w}{L^p} + \frac{1}{p} (\frac{1}{\varepsilon^{1-\alpha}} \norm{\dv \vec w}{L^\infty} + 1) \norm{\rho-M}{L^p} \\
    \lesssim \frac{1}{\varepsilon^{1-\alpha}} \norm{\vec w}{H^3} + \frac{1}{p} (\frac{1}{\varepsilon^{1-\alpha}} \norm{\vec w}{H^3} + 1) \norm{\rho-M}{L^p}.
    \end{gathered}
\end{equation}
Sending $ p \rightarrow \infty $ in \eqref{ene:402} leads to 
\begin{equation}
    \label{ene:403}
    \dfrac{d}{d\tau} \norm{\rho-M}{L^\infty} + \underline\rho/2 \norm{\rho-M}{L^\infty} \lesssim \varepsilon^{\alpha/2} \norm{\varepsilon^{\alpha/2-1}\vec w}{H^3}.
\end{equation}
Let $ \mathfrak{c}_1 \in (0,\min\lbrace \underline\rho/2,\mathfrak c_0/2 \rbrace) $. Then one can derive from \eqref{ene:403} that
\begin{equation}
    \label{ene:404}
    \begin{gathered}
    \norm{\rho(\tau)-M}{L^\infty} \lesssim \varepsilon^{\alpha/2} \int_0^\tau e^{\mathfrak c_1 (s- \tau) } \mathcal D_\mrm{total}^{1/2}(s) \,ds + e^{-\mathfrak{c}_1 \tau} \norm{\rho_0 - M}{L^\infty}\\
    \lesssim \varepsilon^{\alpha/2} \biggl( \int_0^\tau e^{2\mathfrak{c}_1(s-\tau) - \mathfrak{c}_0 s}\,ds \biggr)^{1/2}\biggl( \int_0^\tau e^{\mathfrak c_0 s} \mathcal D_\mrm{total}(s)\,ds \biggr)^{1/2} + e^{-\mathfrak{c}_1 \tau} \norm{\rho_0 - M}{L^\infty} \\
    \lesssim \varepsilon^{\alpha/2} e^{-\mathfrak{c_1}\tau}\mathfrak F^{1/2}(\tau) + e^{-\mathfrak{c}_1 \tau}.
\end{gathered}
\end{equation}

Meanwhile, from \eqref{eq:ptb-01}, one can write down, for $ \partial \in \lbrace \partial_x, \partial_y, \partial_z \rbrace $,
\begin{equation}
    \label{ene:501}
\begin{gathered}
    \dtau \partial \rho + (\frac{1}{\varepsilon^{1-\alpha}}\vec w + \vec v_\rho) \cdot \nabla \partial \rho + \rho \partial \rho + (\rho - M)\partial\rho + \frac{1}{\varepsilon^{1-\alpha}} \partial \rho \dv \vec w \\
    + (\frac{1}{\varepsilon^{1-\alpha}} \partial \vec w + \partial \vec v_\rho) \cdot \nabla \rho + \frac{1}{\varepsilon^{1-\alpha}} \rho \dv \partial \vec w = 0.
\end{gathered}
\end{equation}
Then one can directly calculate that, 
\begin{equation}
    \label{ene:502}
    \begin{gathered}
        \dfrac{d}{d\tau} \norm{\partial\rho}{L^4} + \underline \rho/2 \norm{\partial \rho}{L^4} \lesssim_{\underline\rho,\overline\rho} (\norm{\rho-M}{L^\infty} + \frac{1}{\varepsilon^{1-\alpha}} \norm{\nabla \vec w}{L^\infty} + \norm{\nabla \vec v_\rho}{L^\infty})\norm{\nabla\rho}{L^4}\\
        + \frac{1}{\varepsilon^{1-\alpha}} \norm{\nabla^2 \vec w}{L^4} \stackrel{\eqref{est:KS-map-2}}{\lesssim} (\norm{\rho-M}{L^\infty} + \varepsilon^{\alpha/2} \norm{\varepsilon^{\alpha/2-1}\vec w}{H^3} )\norm{\nabla\rho}{L^4} + \norm{\nabla\rho}{L^4}^2\\
        + \varepsilon^{\alpha/2} \norm{\varepsilon^{\alpha/2-1}\vec w}{H^3}.
    \end{gathered}
\end{equation}
Solving \eqref{ene:502} leads to, for any $ \mathfrak c_2 \in (0,\min\lbrace \underline\rho/2, \mathfrak c_0 /2 \rbrace) $, that for some constant $ \mathfrak c_{012} $ depending only on $ \mathfrak c_0, \mathfrak c_1, \mathfrak c_2 $, 
\begin{equation}
    \label{ene:503}
    \begin{gathered}
    e^{\mathfrak c_2 \tau}\norm{\nabla \rho(\tau)}{L^4} \lesssim \norm{\nabla \rho_0}{L^4} e^{\int_0^\tau (\norm{\rho(s) - M}{L^\infty} + \varepsilon^{\alpha/2}\norm{\varepsilon^{\alpha/2-1}\vec w(s)}{H^3} + \norm{\nabla \rho (s)}{L^4} ) \,ds} \\
    + \varepsilon^{\alpha/2} \int_0^\tau e^{\mathfrak c_2 s + \int_0^\tau (\norm{\rho(s') - M}{L^\infty} + \varepsilon^{\alpha/2}\norm{\varepsilon^{\alpha/2-1}\vec w(s')}{H^3} + \norm{\nabla \rho (s')}{L^4} ) \,ds' } \norm{\varepsilon^{\alpha/2-1}\vec w(s)}{H^3} \,ds \\
    \lesssim \norm{\nabla \rho_0}{L^4} e^{\mathfrak c_{012} ( \varepsilon^{\alpha/2} \mathfrak F^{1/2}(\tau) + 1 + \sup_{0\leq s \leq \tau} (e^{\mathfrak c_2 s} \norm{\nabla\rho(s)}{L^4}) )  } \\
    + \varepsilon^{\alpha/2} e^{\mathfrak c_{012} ( \varepsilon^{\alpha/2} \mathfrak F^{1/2}(\tau) + 1 + \sup_{0\leq s \leq \tau} (e^{\mathfrak c_2 s} \norm{\nabla\rho(s)}{L^4}) )  } \mathfrak F^{1/2}(\tau),
    \end{gathered}
\end{equation}
where we have used the following, thanks to \eqref{ene:400} and \eqref{ene:404}: 
\begin{equation}
    \label{ene:504}
    \begin{gathered}
        \int_0^\tau (\norm{\rho(s) - M}{L^\infty} + \varepsilon^{\alpha/2}\norm{\varepsilon^{\alpha/2-1}\vec w(s)}{H^3} + \norm{\nabla \rho (s)}{L^4} ) \,ds \lesssim \int_0^\tau \biggl( \varepsilon^{\alpha/2} e^{-\mathfrak c_1 s} \mathfrak F^{1/2}(s) \\
        + e^{-\mathfrak c_1 s} + \varepsilon^{\alpha/2} e^{-\mathfrak c_0/2 s} e^{\mathfrak c_0/2 s} \mathcal D_\mrm{total}^{1/2}(s) + e^{-\mathfrak c_2 s} e^{\mathfrak c_2 s} \norm{\nabla\rho(s)}{L^4} \biggr) \,ds \\
        \lesssim_{\mathfrak{c}_{012}} \varepsilon^{\alpha/2} \mathfrak F^{1/2}(\tau) + 1 + \sup_{0\leq s \leq \tau} (e^{\mathfrak c_2 s} \norm{\nabla\rho(s)}{L^4}),
    \end{gathered}
\end{equation}
and 
\begin{equation}
    \label{ene:505}
    \int_0^\tau e^{\mathfrak c_2 s} \norm{\varepsilon^{\alpha/2-1}\vec w(s)}{H^3} \,ds \lesssim \int_0^\tau e^{\mathfrak c_2 s - \mathfrak c_0 s/2} e^{\mathfrak c_0 s/2} \mathcal D_\mrm{total}^{1/2} (s)\,ds \lesssim \mathfrak F^{1/2}(\tau). 
\end{equation}
Consequently, for $ \norm{\nabla \rho_0}{L^4} $ and $ \varepsilon $ small enough, one can conclude from \eqref{ene:503} by the continuity argument that
\begin{equation}
    \label{ene:506}
    \norm{\nabla \rho(\tau)}{L^4} \lesssim (\norm{\nabla \rho_0}{L^4} + \varepsilon^{\alpha/2} \mathfrak F^{1/2}(\tau) ) e^{-\mathfrak c_2 \tau + \mathfrak c_{012} (\varepsilon^{\alpha/2} \mathfrak F^{1/2}(\tau) + 1)}.
\end{equation}

\subsection{Global-in-time estimate}
\label{sec:global-estimate}

Thanks to \eqref{est:202} and \eqref{ene:204}, there exists $ \mathfrak c_0 $ small enough such that 
\begin{equation}
    \label{ene:601}
\begin{gathered}
    \dfrac{d}{d\tau} e^{\mathfrak c_0 \tau} \mathcal E_\mrm{total}(\tau) + e^{\mathfrak c_0 \tau} \mathcal D_\mrm{total}(\tau) \lesssim (\varepsilon^{\alpha/2} + \varepsilon^{2-\alpha}) H(\mathcal E_\mrm{total}(\tau)) e^{\mathfrak c_0 \tau} \mathcal D_\mrm{total}(\tau) \\
        + (\norm{\rho-M}{L^\infty} + \norm{\nabla \rho}{L^4}) e^{\mathfrak c_0 \tau} \mathcal E_\mrm{total}(\tau), 
\end{gathered}
\end{equation}
which implies that, recalling $ \mathfrak F $ in \eqref{ene:400}, for some $ \mathfrak c_4 > 0 $, 
\begin{equation}
    \label{ene:602}
\begin{aligned}
    & \dfrac{d}{d\tau} e^{- \mathfrak c_4 \int_0^\tau (\norm{\rho(s)-M}{L^\infty} + \norm{\nabla \rho(s)}{L^4})\,ds} \mathfrak F(\tau) \\
    &\qquad \lesssim (\varepsilon^{\alpha/2} + \varepsilon^{2-\alpha}) e^{- \mathfrak c_4 \int_0^\tau (\norm{\rho(s)-M}{L^\infty} + \norm{\nabla \rho(s)}{L^4})\,ds} H(\mathfrak F(\tau)) \dfrac{d}{d\tau} \mathfrak F(\tau).
\end{aligned}
\end{equation}

Meanwhile, thanks to \eqref{ene:404} and \eqref{ene:506}, one has that
\begin{equation}
    \label{ene:603}
    \int_0^\tau (\norm{\rho(s)-M}{L^\infty} + \norm{\nabla \rho(s)}{L^4}) \,ds \lesssim (\norm{\nabla\rho_0}{L^4} + \varepsilon^{\alpha/2} \mathfrak F^{1/2}(\tau)) e^{\mathfrak c_{012}(\varepsilon^{\alpha/2}\mathfrak F^{1/2}(\tau) + 1)} + 1.
\end{equation}

Therefore, one can integrating \eqref{ene:602} in $ \tau $ and conclude that 
\begin{equation}
    \label{ene:604}
    \mathfrak F(\tau) \lesssim \bigl\lbrack (\varepsilon^{\alpha/2} + \varepsilon^{2-\alpha}) \mathfrak H(\mathfrak F(\tau)) +  \mathfrak F (0) \bigr\rbrack e^{\mathfrak c_4 \lbrack (\norm{\nabla\rho_0}{L^4} + \varepsilon^{\alpha/2} \mathfrak F^{1/2}(\tau)) e^{\mathfrak c_{012}(\varepsilon^{\alpha/2}\mathfrak F^{1/2}(\tau) + 1)} + 1 \rbrack }. 
\end{equation}
Here $ \mathfrak H(\cdot) $ is the primitive of $ H(\cdot) $, and notice that $ \mathfrak F(\tau) $ is non-decreasing ($ \frac{d}{d\tau} \mathfrak F(\tau) \geq 0 $).

To conclude, for $ \varepsilon $ small enough, applying the continuity argument in \eqref{ene:604} yields 
\begin{equation}
    \label{ene:605}
    \mathfrak F(\tau) \leq \mathfrak F(0) e^{\mathfrak c_4 \norm{\nabla \rho_0}{L^4} e^{\mathfrak c_{012}} + \mathfrak c_4},
\end{equation}
for all $ \tau \in (0,\infty) $ under the {\it a prior} assumption \eqref{ass-prior:rho}. 

\subsection{Closing the \textbf{\it a prior} assumption \eqref{ass-prior:rho}} \label{sec:closing}

Now we are ready to close the {\it a prior} assumption \eqref{ass-prior:rho}. 

For local in time estimate, with \eqref{ene:302}, \eqref{ass-prior:rho} can be closed easily by choosing time small enough. We therefore focus on the global in time case. In fact, we will show a stronger version of \eqref{ass-prior:rho} with the bound \eqref{ene:605}. Then following the continuity argument, one can show both \eqref{ene:605} and \eqref{ass-prior:rho} hold for all $ \tau \in (0,\infty) $. 

\smallskip 

From \eqref{eq:ptb-01}, one can write down 
\begin{equation}
    \label{ene:701}
    \dtau \rho + (\frac{1}{\varepsilon^{1-\alpha}}\vec w + \vec v_\rho) \cdot \nabla \rho  + \rho (\rho - M)  + \frac{1}{\varepsilon^{1-\alpha}} \rho \dv \vec w  = 0.
\end{equation}
Let $ \vec X(\tau,\vec x) $ be the flow trajectory growing from $ \forall \vec x $ defined by
\begin{equation}
    \label{ene:702}
    \begin{cases}
        \partial_\tau \vec X(\tau; \vec x) = (\frac{1}{\varepsilon^{1-\alpha}}\vec w + \vec v_\rho) (\vec X(\tau; \vec x), \tau), \\
        \vec X(0;\vec x)=\vec x.
    \end{cases}
\end{equation}
Then along the flow trajectory, from \eqref{ene:701}, one has that
\begin{equation}
    \label{ene:703}
    \begin{gathered}
    \dfrac{d}{d\tau} \rho(\vec X(\tau; \vec x), \tau ) = \partial_\tau \rho (\vec X(\tau; \vec x), \tau ) + \dt \vec X(\tau;\vec x) \cdot \nabla_{\vec x} \rho (\vec X(\tau; \vec x), \tau ) \\
    \leq - \rho (\vec X(\tau; \vec x), \tau ) (\rho (\vec X(\tau; \vec x), \tau ) - M) + 2 \varepsilon^{\alpha/2} \overline \rho\norm{\varepsilon^{\alpha/2-1}\vec w(\tau)}{H^3}.
\end{gathered}
\end{equation}
Then since $ M < \overline \rho $, after integrating \eqref{ene:703} in the time interval when $ \rho(\vec X(\tau;\vec x),\tau) > M $, one can conclude that
\begin{equation}
    \label{ene:704}
\begin{gathered}
    \rho(\vec X(\tau;\vec x),\tau) \leq \max\lbrace \rho_0(\vec x) + 2 \varepsilon^{\alpha/2} \overline\rho \int_0^\tau  e^{-\mathfrak c_0/2} e^{\mathfrak c_0/2} \norm{\varepsilon^{\alpha/2-1}\vec w(s)}{H^3} \,ds, M \rbrace \\
    \leq \overline \rho + 2 \varepsilon^{\alpha/2} \overline \rho  (\int_0^\infty e^{-\mathfrak c_0s}\,ds)^{1/2} \mathfrak F^{1/2}(\tau),
\end{gathered}
\end{equation}
thanks to \eqref{initial:001} and \eqref{ene:400}. 
Therefore, for $ \varepsilon $ small enough, together with \eqref{ene:605}, \eqref{ene:704} implies that for $ \forall (\vec x, \tau) $, 
\begin{equation}
    \label{ene:705}
    \rho(\vec x, \tau) < \frac{3}{2} \overline\rho . 
\end{equation}

To show the lower bound of $ \rho $, one can derive from \eqref{ene:701} that, similar to \eqref{ene:703}
\begin{equation}
    \label{ene:706}
    \begin{gathered}
        \dfrac{d}{d\tau} \rho(\vec X(\tau; \vec x), \tau ) = \partial_\tau \rho (\vec X(\tau; \vec x), \tau ) + \dt \vec X(\tau;\vec x) \cdot \nabla_{\vec x} \rho (\vec X(\tau; \vec x), \tau ) \\
        \geq - \rho (\vec X(\tau; \vec x), \tau ) (\rho (\vec X(\tau; \vec x), \tau ) - M) - \frac{1}{2}\varepsilon^{\alpha/2} \underline \rho\norm{\varepsilon^{\alpha/2-1}\vec w(\tau)}{H^3}.
    \end{gathered}
\end{equation}
Then since $ \underline \rho < M $, integrating \eqref{ene:706} in the time interval when $ \rho(\vec X(\tau;\vec x),\tau) < M $ yields that 
\begin{equation}
    \label{ene:707}
    \begin{gathered}
        \rho(\vec X(\tau;\vec x),\tau) \geq \min\lbrace \rho_0(\vec x) - \frac{1}{2} \varepsilon^{\alpha/2} \underline\rho \int_0^\tau  e^{-\mathfrak c_0/2} e^{\mathfrak c_0/2} \norm{\varepsilon^{\alpha/2-1}\vec w(s)}{H^3} \,ds, M \rbrace \\
    \geq \underline \rho - \frac{1}{2} \varepsilon^{\alpha/2} \underline \rho  (\int_0^\infty e^{-\mathfrak c_0s}\,ds)^{1/2} \mathfrak F^{1/2}(\tau),
    \end{gathered}
\end{equation}
thanks to \eqref{initial:001} and \eqref{ene:400}.
Then thanks to \eqref{ene:605}, for $ \varepsilon $ small enough, one can conclude, together with \eqref{ene:705}, that
\begin{equation}
    \label{ene:708}
    \frac{3}{4} \underline \rho \leq \rho(\vec x, \tau) \leq \frac{3}{2} \overline \rho, 
\end{equation}
which is stronger than \eqref{ass-prior:rho} and thus finishes the proof.

\section{Large friction limit}
\label{sec:convergence}

Without loss of generality, let $ T $ be the existence time of $ (\rho, \vec w) $, satisfying the following uniform-in-$\varepsilon $ bound:
\begin{equation}
    \label{ene:801}
    \begin{aligned}
        \sup_{0\leq \tau \leq T} \bigl( \norm{\varepsilon^{\alpha/2} \vec w(\tau)}{H^3}^2 + \norm{\rho(\tau)-M}{H^3}^2 \bigr) + \int_0^T \bigl( \norm{\varepsilon^{\alpha/2-1}\vec w(s)}{H^3}^2 + \norm{\rho(s)-M}{H^3}^2 \bigr) \,ds < \infty. 
    \end{aligned}
\end{equation}
Then directly using equations \eqref{eq:ptb-01} and \eqref{eq:ptb-02}, thanks to \eqref{ass-prior:rho}, one has that
\begin{equation}
    \label{ene:802}
    \begin{gathered}
    \sup_{0 \leq \tau \leq T} \bigl( \norm{\varepsilon^{1-\alpha/2}\dtau\rho(\tau)}{H^2}^2 + \norm{\varepsilon^{\alpha/2+2}\dtau \vec w(\tau)}{H^2}^2 \bigr) \\
    + \int_0^T \bigl( \norm{\dtau\rho(s)}{H^2}^2 + \norm{\varepsilon^{\alpha/2 + 1}\dtau \vec w(s)}{H^2}^2 \bigr) \,ds < \infty.
    \end{gathered}
\end{equation}
We remind readers that the estimate \eqref{ene:802} implies the existence of a (stronger) initial layer for $ \vec w $ and a (weaker) initial layer for $ \rho $. Therefore, in general, one should not expect strong convergence in system \eqref{sys:PLEP-perturbation}. 

Applying the Aubin-Lions compactness lemma, one can obtain from \eqref{ene:801} and \eqref{ene:802} that as $ \varepsilon \rightarrow 0^+ $,
\begin{align}
    \label{ene:803}
    \rho - M  \stackrel{*}{\rightharpoonup} & ~ \sigma - M && \text{in} ~ L^\infty(0,T;H^3);\\
    \label{ene:804}
    \rho - M {\rightharpoonup} & ~ \sigma - M && \text{in} ~ L^2(0,T;H^3); \\
    \label{ene:805}
    \rho - M \rightarrow & ~ \sigma - M && \text{in} ~ C(0,T;H^2_\mrm{loc}); \\
    \label{ene:806}
    \dtau \rho \rightharpoonup & ~ \dtau \sigma && \text{in} ~ L^2(0,T;H^2);
\end{align}
for some $ \sigma - M \in L^\infty(0,T;H^2)\cap L^2(0,T;H^3) $ with $ \dtau \sigma \in L^2(0,T;H^2) $. 
Meanwhile, let 
\begin{equation}
    \label{ene:807}
    \vec v := - \nabla (-\Delta)^{-1} (\sigma-M). 
\end{equation}
Then thanks to \eqref{def:KS-map} and \eqref{ene:802}, one has that
\begin{align}
    \label{ene:808}
    \vec v_\rho \rightharpoonup & ~ \vec v && \text{in} ~ L^2(0,T;H^4); \\
    \label{ene:809}
    \vec v_\rho \rightarrow & ~ \vec v && \text{in} ~ C(0,T;H^3_\mrm{loc}).
\end{align}

Moreover, one can calculate that, thanks to \eqref{ene:801}, 
\begin{equation}
    \label{ene:810}
    \begin{gathered}
        \int_0^T \norm{\dtau \rho + \dv(\rho \vec v_\rho)}{H^2}^2 \,ds = \int_0^T \norm{\frac{1}{\varepsilon^{1-\alpha}}\dv(\rho \vec w)}{H^2}^2 \,ds \\
        \lesssim \varepsilon^{\alpha} \int_0^T \norm{\varepsilon^{\alpha/2-1} (\vec w \cdot \nabla \rho + (\rho - M) \dv \vec w + M \dv \vec w)}{H^2}^2 \,ds \\
        \lesssim \varepsilon^{\alpha} (\sup_{0\leq \tau \leq T }\norm{\rho(\tau) - M}{H^3}^2 + 1) \int_0^T \norm{\varepsilon^{\alpha/2-1}\vec w(s)}{H^3}^2\,ds  \rightarrow 0
    \end{gathered}
\end{equation}
as $ \varepsilon \rightarrow 0 $. 
Therefore, passing the limit $ \varepsilon \rightarrow 0 $ in \eqref{eq:ptb-01} in the distribution sense yields that
\begin{subequations}
    \label{sys:limit}
\begin{equation}
    \label{ene:811}
    \dtau \sigma + \dv ((\sigma - M) \vec v + M \vec v) = 0 \qquad \text{in} ~ \Omega \times (0,T),
\end{equation}
with
\begin{equation}
    \label{ene:812}
    \vec v = - \nabla (-\Delta)^{-1} (\sigma-M),
\end{equation}
\end{subequations}
which shows \eqref{sys:PLEP-limit}, or equivalently, \eqref{eq:rsPLEPlt-limit}.

\section{One dimensional flow on $ \mathbb R $}
\label{sec:1-d-discussion}

\subsection{Uniform-in-$\varepsilon $ estimate for the one dimensional analogue}
\label{sec:1-d-uniform-est}

The one dimensional analogue of system \eqref{sys:PLEP-perturbation} is given by
\begin{subequations}
    \label{sys:1-d}
    \begin{align}
        \label{eq:1-d-01}
        \dtau \rho + \frac{1}{\varepsilon^{1-\alpha}} \dx (\rho w) + \dx(\rho v_\rho) = & 0, \\
        \label{eq:1-d-02}
        \varepsilon^{\alpha} \rho \dtau w + \varepsilon^{\alpha-1} \rho u \dx w + \frac{1}{\varepsilon^{1-\alpha}} \dx p(\rho) + \frac{1}{\varepsilon^{2-\alpha}} \rho w = & - \varepsilon\rho \dtau v_\rho - \rho u \dx v_\rho,
    \end{align}
    with
    \begin{align}
        \label{eq:1-d-03}
        v_\rho(x,t) =& - \dx (-\partial_{xx})^{-1} (\rho - M)  = \int_{-\infty}^{x} (\rho(y,t)-M)\,dy,\\
        \label{eq:1-d-04}
        u(x,t) = & \varepsilon v_\rho(x,t) + \varepsilon^\alpha w(x,t). 
    \end{align}
\end{subequations}

For the one-dimensional system \eqref{sys:1-d}, one can obtain the uniform-in-$ \varepsilon $ estimate without the smallness condition of $ \norm{\nabla \rho_0}{L^4} $ in \ref{thm:uniform-2}. 

Indeed, recall that the condition of $ \nabla \rho_0 $ is to control the quadratic-in-$ \nabla \rho $ term in \eqref{ene:502}, such that \eqref{ene:506} holds. This is necessary thanks to \eqref{est:101-07-5}, where we have used the estimate $ \norm{\partial \vec v_\rho}{L^\infty} \lesssim \norm{\nabla\rho}{L^4}$. Therefore, the estimate of $ \norm{\nabla\rho}{L^4} $ is a key ingredient for the global-in-time estimate for the three-dimensional problem in section \ref{sec:global-estimate}. 

However, for the one dimensional flow \eqref{sys:1-d}, since $ \partial_x v_\rho = \rho - M $, where $ v_\rho(x,t) := \int_{-\infty}^{x} (\rho(y,\tau) - M)\,dy $, one will only need to control $ \norm{\partial_x v_\rho}{L^\infty} = \norm{\rho - M}{L^\infty} $, and will not need the control of $ \norm{\nabla\rho}{L^4} $ in \eqref{est:101-07-5}, and therefore $ \norm{\nabla \rho}{L^4} $ will not appear in \eqref{ene:204}. Namely, the exponential decay of $ \norm{\rho - M}{L^\infty} $ is sufficient to guarantee the global estimate, repeating the calculation as in section \ref{sec:global-estimate}. Moreover, $ H^2 $ regularity instead of $ H^3 $ is sufficient to close the estimate. 

\subsection{Asymptotic behavior using characteristic method}
\label{sec:asymptotic-behavior-1-d} 

\subsubsection*{The one dimensional hyperbolic-elliptic KS system of consumption type}
\label{sec:1-d-ks}

Sending $ \varepsilon \rightarrow 0 $ in system \ref{sys:1-d}, one will obtain the one dimensional analogue of system \eqref{sys:limit}, which reads 
\begin{subequations}
    \label{sys:1-d-limit}
    \begin{equation}
        \dtau \sigma + \dx(\sigma v) = 0, \label{eq:1-d-limit-01}\\
    \end{equation}
    with
    \begin{align}
        v(x,\tau) =& \int_{-\infty}^{x} (\sigma(y,\tau) - M) \,dy, \label{eq:1-d-limit-02} \\
        \sigma(x,0) =& \sigma_0(x),  \label{eq:1-d-limit-03}
    \end{align}
    satisfying 
    \begin{equation}
        \label{eq:1-d-limit-04}
        \int_{-\infty}^\infty (\sigma_0 - M) \,dx = 0. 
    \end{equation}
\end{subequations}
Equivalently, one can rewrite \eqref{eq:1-d-limit-01} as 
\begin{equation}{\tag{\ref{eq:1-d-limit-01}'}}
    \label{eq:1-d-limit-05}
    \dtau \sigma + v \dx \sigma + \sigma(\sigma - M) = 0.
\end{equation}
Then one can see that \eqref{eq:1-d-limit-05} is an ODE along the trajectory given by the transport velocity $ v $, and there are a stable equilibrium $ \sigma_{e,1} = M $ and an unstable equilibrium $ \sigma_{e,2} = 0 $.

In the case when $  \sigma_0 $ is strictly positive, then along each trajectory, $ \sigma $ will converge to the stable equilibrium $ \sigma_{e,1} = M $. This asymptotic behavior is consistent with the uniform-in-$\varepsilon$ estimate discussed in section \ref{sec:1-d-uniform-est}.

However, if $ \sigma_0 $ is not strictly positive, the regularity of the global solution will not in general be asymptotically bounded. To describe this phenomenon, without loss of generality, consider initial data $ \sigma_0 \in C^1(\mathbb R) $ satisfying  
\begin{equation}
    \label{initial-1-d}
    \begin{aligned}
        \sigma_0 > 0 & && \text{for} \ x < 0, \\
        \sigma_0 = 0 & && \text{for} \ 0 \leq x \leq 1, \\
        \sigma_0 > 0 & && \text{for} \ x > 1,
    \end{aligned}
\end{equation}
where $ (0, 1) \subset \mathbb R $ is the interval of initial vacuum. 

\subsubsection*{Asymptotic behavior along the flow trajectory and the shrinking of the vacuum interval}

Let $ \eta(x,\tau), \ \forall x \in \mathbb{R} $, be the flow trajectory evolving from $ x $; that is
\begin{equation}
    \label{ene:901}
    \begin{cases}
        \dtau \eta(x,\tau) = v(\eta(x,\tau),\tau) = \int_{-\infty}^{\eta(x,\tau)} (\sigma(y, \tau) - M)\,dy, \\
        \eta(x,0) = x.
    \end{cases}
\end{equation}
Then the flow velocity satisfies
\begin{equation}
    \label{ene:902}
\begin{gathered}
    \dfrac{d}{d\tau} v(\eta(x,\tau),\tau) = (\sigma(\eta(x,\tau),\tau) - M) \dtau \eta(x,\tau) + \int_{-\infty}^{\eta(x,\tau)} \dtau \sigma(y,\tau) \,dy\\
    \stackrel{\eqref{eq:1-d-limit-01}}{=} (\sigma(\eta(x,\tau),\tau) - M) v(\eta(x,\tau),\tau) - \sigma(\eta(x,\tau),\tau) v(\eta(x,\tau),\tau) = - M v(\eta(x,\tau),\tau).
\end{gathered}
\end{equation}
This implies that 
\begin{gather}
    \label{ene:903}
    v(\eta(x,\tau),\tau) = e^{-M\tau} v(x,0) = e^{-M\tau} \int_{-\infty}^{x} (\sigma_0(y) - M) \,dy, \\
    \label{ene:903-1} 
    \eta(x,\tau) = x + \int_0^\tau v (\eta(x,s),s) \,ds = x + \frac{1}{M}(1 - e^{-M\tau})\int_{-\infty}^x(\sigma_0(y) - M)\,dy , \\
    \label{ene:904}
    \text{and} \qquad \lim_{\tau\rightarrow \infty} \eta(x,\tau) = x + \frac{1}{M} \int_{-\infty}^{x} (\sigma_0(y) - M)\,dy.
\end{gather}
Moreover, \eqref{eq:1-d-limit-05} can be written as,
\begin{equation}
    \label{ene:905}
    \dfrac{d}{d\tau} \sigma(\eta(x,\tau),\tau) = - \sigma(\eta(x,\tau),\tau) (\sigma(\eta(x,\tau),\tau)- M),
\end{equation}
which is globally well-posed for any fixed $ x $, and 
\begin{equation}
    \label{ene:906}
    \lim_{\tau\rightarrow \infty} \sigma(\eta(x,\tau),\tau) =
    \begin{cases}
    M & \text{if} \ \sigma_0(x) > 0, \\
    0 & \text{if} \ \sigma_0(x) = 0.
    \end{cases}
\end{equation}
In the case when $ \sigma_0(x) > 0 $, i.e., along the {\bf non-vacuum trajectory}, one has furthermore that 
\begin{equation}
    \label{ene:907}
    \sigma(\eta(x,\tau),\tau) \geq \min\lbrace\sigma_0(x),M \rbrace \qquad \text{and} \qquad
    \vert \sigma(\eta(x,\tau),\tau) - M \vert \leq e^{-\min\lbrace \sigma_0(x),M\rbrace \tau} \vert \sigma_0(x) - M \vert;
\end{equation}
In the case when $ \sigma_0(x) = 0 $, i.e., along the {\bf vacuum trajectory}, one has
\begin{equation}
    \label{ene:907-1}
    \sigma(\eta(x,\tau),\tau) \equiv 0.
\end{equation}

In particular, let $ (a(\tau), b(\tau)) $ be the {\bf interval of vacuum} for $ \tau \geq 0 $ for the solution with initial data \eqref{initial-1-d}; that is
\begin{equation}
    \label{ene:908}
    a(\tau) = \eta(0,\tau), \qquad b(\tau) = \eta(1,\tau).
\end{equation}
One can conclude from \eqref{ene:903-1} and \eqref{ene:904} that
\begin{gather}
    \label{ene:909}
    b(\tau) - a(\tau) = e^{-M\tau}, \\
    \label{ene:910}
    \text{and} \qquad \lim_{\tau \rightarrow \infty} a(\tau) = \lim_{\tau \rightarrow \infty} b(\tau) = \frac{1}{M} \int_{-\infty}^0 (\sigma_0(y) - M )\,dy,
\end{gather}
namely, {\bf the interval of vacuum will shrink to a (finite) point exponentially fast} as $ \tau \rightarrow \infty $. 

\subsection*{$ \dx \sigma $ along the flow trajectory}

Finally, we will investigate the asymptotic of $ \dx\sigma $. From \eqref{ene:901}, one can write 
\begin{equation}
    \label{ene:911}
    \begin{cases}
        \dtau \dx \eta(x,\tau) = (\sigma(\eta(x,\tau),\tau) - M) \dx \eta(x,\tau),\\
        \dx \eta(x,0) = 1.
    \end{cases}
\end{equation}
We first discuss the evolution {\bf along the non vacuum trajectory}, i.e., for fixed $ x $ such that $ \sigma_0(x) > 0 $.
Thanks to \eqref{ene:907}, one has that 
\begin{equation}
    \label{ene:912}
    \begin{gathered}
    \log \vert \dx \eta(x,\tau) \vert \leq \int_0^\tau e^{-\min\lbrace \sigma_0(x),M\rbrace s} \vert \sigma_0(x)-M\vert \,ds \\ =\frac{1}{\min\lbrace \sigma_0(x),M\rbrace}(1-e^{-\min\lbrace \sigma_0(x),M\rbrace \tau})\vert \sigma_0(x) - M \vert 
    \stackrel{\tau \rightarrow \infty}{\rightarrow} \frac{1}{\min\lbrace \sigma_0(x),M\rbrace}\vert \sigma_0(x) - M \vert.
    \end{gathered}
\end{equation}
In particular, this implies that, there exists some $ C_x \in (0,\infty) $, such that
\begin{equation}
    \label{ene:913}
    \frac{1}{C_x} < \dx \eta(x,\tau) < C_x < \infty,
\end{equation}
for all $ \tau > 0 $. 

Meanwhile, let 
\begin{equation}
\label{ene:914}
\xi(x,\tau) := \dfrac{d}{dx} \sigma (\eta(x,\tau),\tau) = \dx \eta(x,\tau) \dx\sigma(\eta(x,\tau),\tau).
\end{equation}
Then one can write down from \eqref{ene:905} that 
\begin{equation}
    \label{ene:915}
    \dfrac{d}{d\tau} \xi(x,\tau) = - (2\sigma(\eta(x,\tau),\tau)- M) \xi(x,\tau).
\end{equation}
Thanks to \eqref{ene:907}, one can conclude from \eqref{ene:915} that
\begin{equation}
    \label{ene:916} 
    \lim_{\tau \rightarrow \infty} \xi(x,\tau) = 0.
\end{equation}
Therefore, 
\begin{equation}
    \label{ene:917}
    \dx \sigma(\eta(x,\tau),\tau) = \frac{1}{\dx \eta(x,\tau)}\dfrac{d}{dx} \sigma(\eta(x,\tau),\tau) = \frac{1}{\dx\eta(x,\tau)}  \xi(x,\tau) \stackrel{\tau\rightarrow \infty}{\rightarrow} 0.
\end{equation}

\smallskip 

Now we discuss the evolution {\bf along the vacuum trajectory}, i.e., for fixed $ x $ such that $ \sigma_0(x) = 0 $. Indeed, \eqref{ene:911} is then reduced to 
\begin{equation}
    \label{ene:918}
    \dtau\dx\eta(x,\tau) = - M \dx \eta(x,\tau), \qquad \dx \eta(x,0) = 1, 
\end{equation}
which implies that 
\begin{equation}
    \label{ene:919}
    \dx \eta(x,\tau) = e^{-M\tau}. 
\end{equation}
Meanwhile, \eqref{ene:915} is reduced to
\begin{equation}
    \label{ene:920}
    \dfrac{d}{d\tau} \xi(x,\tau) = M \xi(x,\tau). 
\end{equation}
Therefore, one has that 
\begin{equation}
    \label{ene:921}
    \xi(x,\tau) = e^{M\tau} \xi(x,0) = e^{M\tau}\dx \sigma_0(x).
\end{equation}
Consequently, one has that, similar to \eqref{ene:917}, 
\begin{equation}
    \label{ene:922}
    \dx\sigma(\eta(x,\tau),\tau)= \frac{\xi(x,\tau)}{\dx \eta(x,\tau)} = e^{2M\tau} \dx \sigma_0(x).
\end{equation}
This implies that the $ \norm{\sigma(\tau)}{H^2} $ is growing exponentially in time thanks to the embedding inequality $  \norm{\sigma(\tau)}{H^2}  \gtrsim \norm{\dx \sigma(\tau)}{L^\infty} \simeq e^{2M\tau}  $, provided there exists some point such that $ \dx \sigma_0(x) \neq 0 $ and $ \sigma_0(x) = 0 $ ($ x = 0 $ or $ 1 $ for the initial data in \eqref{initial-1-d} for instance). 

This, in particular, implies that one {\bf should not expect global uniform-in-$\tau$ Sobolev regularity of solutions to the limit system \eqref{sys:1-d-limit} in the case when there exists vacuum initially}. In other words, condition \eqref{initial:001} is necessary for theorem \ref{thm:uniform} and theorem \ref{thm:limit}.

\subsubsection*{High order derivatives of $ \sigma $ along the flow trajectory}

In general, assume that the first nontrivial derivative of $ \sigma_0(x), \ x = 0,1 $ is the $ k$-th derivative, for some $ k \in \mathbb Z^+ $, i.e., 
\begin{equation}
    \label{ene:923}
    \begin{gathered}
    \partial_{x}^j \sigma_0(x) = 0, \quad \forall \ j \leq k - 1, \\
    \partial_x^k \sigma_0(x) \neq 0.
    \end{gathered}
\end{equation}
Such $ k $ always exists for nontrivial data. 
Then from \eqref{ene:905}, one has that
\begin{equation}
    \label{ene:924}
    \begin{gathered}
    \dfrac{d}{d\tau} \vec V_k(x,\tau ) = A(\vec V_k(x,\tau))\vec V_k(x,\tau), \\
    \text{with} \qquad \vec V_k(x,0) = 0,
\end{gathered}
\end{equation}
where $ \vec  V_k(x,\tau) := (\sigma(\eta(x,\tau),\tau), \dfrac{d}{dx}\sigma(\eta(x,\tau),\tau), \cdots, \dfrac{d^{k-1}}{dx^{k-1}}\sigma(\eta(x,\tau),\tau))^\top $ is the tube of derivatives with orders smaller than $ k $, and $ A(\vec V_k(x,\tau)) $ is a non-singular matrix. Then one can derive that, 
\begin{equation}
    \label{ene:925}
    \vec V_k(x,\tau) \equiv 0.
\end{equation}
That is, after applying the chain rule,
\begin{equation}
    \label{ene:926}
    0 \stackrel{\eqref{ene:925}}{=} \dfrac{d^j}{dx^j} \sigma(\eta(x,\tau),\tau) = \partial_x^j \sigma(\eta(x,\tau),\tau) (\dx \eta(x,\tau))^j + \mrm{l.o.t}, \quad \forall \ j \leq k - 1.
\end{equation}
Here, $ \mrm{l.o.t} $ denotes the terms involving only derivatives with orders smaller than $ j $. 
Thanks to \eqref{ene:919}, one can derive from \eqref{ene:926} by induction that
\begin{equation}
    \label{ene:926-1}
    \dx^j \sigma (\eta(x,\tau),\tau) \equiv 0, \qquad \forall j \leq k-1.
\end{equation}

Meanwhile, for the $ k $-th derivative, thanks to \eqref{ene:926}, one can calculate from \eqref{ene:905} that 
\begin{equation}
    \label{ene:927}
    \dfrac{d}{d\tau} \dfrac{d^k}{dx^k} \sigma(\eta(x,\tau),\tau) = M \dfrac{d^k}{dx^k} \sigma(\eta(x,\tau),\tau).
\end{equation} 
Consequently, 
\begin{equation}
    \label{ene:928}
    e^{M\tau} \dx^k \sigma_0(x) \stackrel{\eqref{ene:927}}{=} \dfrac{d^k}{dx^k} \sigma(\eta(x,\tau),\tau) = \partial_x^k \sigma(\eta(x,\tau),\tau) (\dx \eta(x,\tau))^k + \mrm{l.o.t}. 
\end{equation}
Here, $ \mrm{l.o.t} $ denotes the terms involving only derivatives with orders smaller than $ k $. Thanks to \eqref{ene:919} and \eqref{ene:926-1}, one can conclude from \eqref{ene:928} that
\begin{equation}
    \label{ene:929}
    \dx^k\sigma(\eta(x,\tau),\tau) = e^{(k+1)M\tau} \dx^k\sigma_0(x).
\end{equation}

\section*{Data availability statement}
No data is available for this paper.




\end{document}